\documentclass{article}
\usepackage[a4paper]{geometry}
\usepackage[utf8]{inputenc}

\usepackage[colorlinks]{hyperref}
\usepackage[singlelinecheck=false]{caption}
\usepackage{booktabs}

\usepackage{graphicx}
\usepackage{latexsym}

\usepackage{lipsum}
\usepackage{graphicx}
\usepackage{authblk}
\usepackage[english]{babel}
\usepackage{color}
\usepackage{hyperref}
\hypersetup{
    colorlinks=true,
    linkcolor=blue,
    filecolor=blue,      
    urlcolor=blue,
    citecolor=blue,
    }
\usepackage{dsfont}
\usepackage{geometry}
\usepackage{float}
\usepackage{bbm}
\usepackage{enumerate}

\usepackage[nottoc]{tocbibind}
 \usepackage{indentfirst}

\graphicspath{{./figure/}}

\usepackage{pgfplots}
\usepackage{comment}
\usepackage{fontenc}
\usepackage{booktabs}
\usepackage{tabularx}
\usepackage{tikz}
\usepackage{caption}
\usepackage{subcaption}
\usepackage{comment}

\DeclareCaptionLabelFormat{custom}
{%
      \textbf{Figure #2}
}
\DeclareCaptionLabelSeparator{custom}{ }
\DeclareCaptionFormat{custom}
{%
\centering
    #1 #2 #3 
}
 \usepackage{tikz-cd}

\usepackage{amsmath,amsfonts,amssymb,amsthm}

\numberwithin{equation}{section}
\usepackage{etoolbox}
\apptocmd{\thebibliography}{}{}{}
\theoremstyle{plain}
\newtheorem{theorem}{Theorem}[section]
\newtheorem{lemma}[theorem]{Lemma}

\newtheorem{proposition}[theorem]{Proposition}

\newtheorem{reminder*}{[theorem]Reminder}

\newtheorem{details*}[theorem]{Details}

\newtheorem{comm*}{Comment}
\newtheorem{example}[theorem]{Example}
\newtheorem{definition}[theorem]{Definition} 
\newtheorem{definition*}{[theorem]Definition}

\newtheorem{notation*}{Notation}
\newtheorem{assumption}[theorem]{Assumption}

\title{The adaptation property in non-equilibrium chemical systems}
  \author{E. Franco, J. J. L. Velázquez}

\begin{document}

\maketitle

\begin{abstract}
    The goal of this paper is to understand if the property of adaptation, which is a typical property of many biochemical systems, can be achieved only by biological systems that actively consume energy or if it can be achieved also by passive systems. 
   We prove that, unless the conserved quantities of a signalling system satisfy a very specific factorization assumption, adaptation cannot be achieved in a robust manner (i.e. in a stable manner under perturbations of the chemical reaction rates) by systems that satisfy the detailed balance property, hence that are passive, and that do not exchange substances or energy with the environment. 
   We also prove that robust adaptation can be achieved by systems that satisfy the detailed balance property, but exchange substances with the environment.
\end{abstract}


\tableofcontents

\section{Introduction}
Many biochemical networks that can be found in the biological literature do not satisfy the detailed balance property.
In this sense, these biochemical networks operate in ”out of equilibrium” conditions and must spend energy in order to function. Many systems for which the detailed balance condition is not satisfied can be found, for instance, in \cite{alon2019introduction,milo2015cell,phillips2012physical}. 
Typical examples of systems that do not satisfy the detailed balance property are the models of chemotaxis and the kinetic proofreading models (see \cite{franco2024stochastic,hopfield1974kinetic}). 

It is then natural to ask which functions of a biological system require lack of equilibrium and which ones can be performed in equilibrium conditions. 
In other words, it is interesting to understand which biological functions require an active use of energy and/or fluxes of substances from/into the environment and which biological functions of the chemical network, instead, can take place in a passive manner.

In this paper we will consider in particular the so called \textit{adaptation} property. This property consists in the fact that many biochemical systems react to changes in a signal more than to the absolute values of the signal concentration. In other words, many biological systems respond to a stimulus, but after a transient time they return to the basal activity level that they had before the stimulus. 
One of the goals of this paper is to understand if the property of adaptation requires that the system is "out of equilibrium" (i.e. it spends energy) or if adaptation can be achieved in systems satisfying the detailed balance condition.
Notice that a systematic analysis of the network structures that yield the adaptation property has been done in \cite{ferrell2016perfect,ma2009defining}. In particular in 
\cite{ma2009defining} a computational analysis of a large class of networks satisfying the adaptation property has been performed. 
In \cite{ferrell2016perfect} several types of "logical structures" including activatory and inhibitory mechanisms yielding adaptation have been studied and it has been examined under which conditions adaptation can take place.

A classical example of biological model that exhibits the adaptation property is the model of bacterial chemotaxis studied by Barkai and Leibler (see \cite{barkai1997robustness}). Other biological systems that satisfy the adaptation property are many sensory systems, as for instance the visual sensory system (see \cite{segel1986mechanism,walz1987consequences}). 
Another example of model in which adaptation plays a role is mitogen signalling, i.e. the mechanism that induces or enhances cell division. The adaptation property in this case can avoid uncontrolled proliferation. 

The goal of this paper is to study the relation between detailed balance and the adaptation property described above.
To this end, we define a class of ideal \textit{signalling systems}. 
These are kinetic systems in which the concentration of one substance, that we call \textit{signal}, changes in time due to influxes and outfluxes of signal.
Therefore, the signal does not follow the dynamics imposed by the reactions in the kinetic system, instead its evolution is prescribed by some external boundary condition. We will denote the concentration of signal as $f$.
 Let $\Omega=\{ 1, \dots, N \} $ be the set of the substances in the network, then the change in time of the vector of the concentrations $n(t)=(n_1(t), \dots, n_N(t)) $ of the substances in the network is given by 
\begin{align}\label{eq:intro signalling}
    \frac{d n_i}{ dt} (t)&= J_i(n), \quad i \in \{ 2, \dots, N \} \\
  \frac{d n_1}{ dt} (t)&= J_1(n)+ J^{ext}(n) .  \nonumber
\end{align}
Here $J_i $ represent the net flux of chemicals at state $i \in \Omega$ due to the chemical reactions taking place in the network. Instead $J^{ext}(n)$ is the net flux of signal, i.e. of substance $(1)$, coming from the environment. In particular we assume that $ J^{ext}(n) $ is such that $n_1(t)=f(t) $. Hence the ODEs describing the change in time of the vector of the concentration in the  signalling system can be equivalently formulated as 
\begin{align*}
    \frac{d n_i}{ dt} (t)&= J_i(n), \quad i \in \{ 2, \dots, N \} \\
 n_1(t)&=f(t).  \nonumber
\end{align*}
If we ignore the external fluxes in \eqref{eq:intro signalling} we obtain a system of ODEs describing the change in time of the concentrations of substances in a kinetic system, i.e. 
\begin{equation} \label{eq:intro kinetic}
    \frac{d n_i}{ dt} (t)= J_i(n), \quad i \in \{ 1, \dots, N \}. 
\end{equation}
We refer to the kinetic system corresponding to \eqref{eq:intro kinetic} as the \textit{underlying kinetic system of a signalling system}. 

In this work we assume that the concentration of signal tends to constant values as time tends to infinity. 
It is reasonable to expect that changes in the concentration of the signal yield changes in some of the concentrations of other substances in the signalling system. In particular, the aim is to study how the concentration of a specific substance changes in time after changes in the signal concentration.
We refer to this substance as the \textit{product} of the signalling system and the change in the concentration of product is the response of the signalling system to the changes in the signal concentration. 

The signalling systems that we consider in this paper are simple models that mimic the usual structure of biological signalling systems. 
One example of simple signalling system is the two-component signalling pathway. This pathway consists of two steps, as a first step a ligand binds to a cell receptor. The complex ligand receptor then undergoes a sequence of chemical reactions, that usually requires consumption of  molecules (for instance ATP), and finally forms a product (we refer to \cite{milo2015cell} for more details on these models).

In this paper we will study signalling systems that are in contact with a reservoir of temperature and, as a consequence, whose temperature is constant. Moreover, we assume that the signalling systems that we study are endowed with mass action kinetics. 
The reason for this assumption is that other types of kinetics, as the Michaelis-Menten kinetics, can be derived as limit of systems endowed with mass action kinetics (see \cite{goldbeter1981amplified,segel1989quasi}).
In Section \ref{sec:examples} we will discuss a limit of this type for a particular example of model of adaptation.

The reason why we introduce signalling systems is because we want to define the adaptation property for these types of systems. More precisely, we say that a signalling system satisfies the adaptation property if the system reacts to changes in the signal, but, after a transient time, the concentration of product returns to the pre-signal value (i.e. the concentration of product that we would have at the steady state if we ignore the influxes and outfluxes of chemicals, i.e the steady states of \eqref{eq:intro kinetic}).  See Figure \ref{fig1} for a graphical representation of the adaptation property. 
As we will see later in Section \ref{sec:adaptation}, in order to study the relation between the detailed balance property and the adaptation property it is important also to study the form of the conserved quantities of the system that we are analysing. In this context, we say that a kinetic system has a conserved quantity if the sum of a certain number of molecules, or a set of functional groups composed by several substances, remains constant during all the  reactions taking place in the network.  
In this paper we say that a signalling system is \textit{closed} if the underlying 
kinetic system satisfies the detailed balance property and it does not exchange substances with the environment (and therefore, there exists a conserved quantity involving all the substances in the network). 

Notice that in order to satisfy the adaptation property a kinetic system must satisfy two main properties: the signalling system must react to changes in the signal and after a transient time should return to the original state.
We study these two properties in detail. 
Given a chemical network we can generate a graph having as vertices all the substances of the chemical network and where the substances belonging to the same reactions are connected. If this graph is connected, then there is a non trivial response to changes in the signal concentration values, unless the chemical rates characterizing the chemical network are fine tuned (i.e. the chemical rates are contained in manifolds with lower dimensionality compared to the one of the space of parameters).  
This allows us to prove that, if a kinetic system is connected, it satisfies the detailed balance property and if there exists a substance that does not appear in any conserved quantity, i.e. the kinetic system exchanges chemicals with the environment, then the adaptation property holds unless the parameters are fine tuned. 
Therefore, this gives a mechanism that allows to obtain robust adaptation (i.e. the adaptation property remains valid under small changes of the chemical rates) when the detailed balance property holds, but the  system is open. 

On the contrary, if we assume that the kinetic system satisfies the detailed balance property and there is no exchange of substances between the kinetic system and the environment (i.e. if we assume that the kinetic system is closed) and the conserved quantities satisfy a suitable non-factorization property, then the adaptation property does not hold, unless the parameters are fine tuned.
The factorization condition mentioned above imposes a very strong constraint to the form of the conserved quantities as it implies that it is possible to separate the substances of the system in classes that do not exchange chemicals via the reactions taking place in the network. 
As we will explain in Section \ref{sec:adaptation in closed systems} the factorization property of the conserved quantity is unlikely to happen in real biological systems. 
Summarizing, in this paper we prove that a kinetic system must exchange mass or energy with the environment in order to satisfy the adaptation property in a robust manner, hence the adaptation property holds also if we perturb the chemical rates. 

It is worth to mention here that the possibility of having adaptation in chemical systems satisfying the detailed balance property was considered in \cite{segel1986mechanism} and the analysis in there was extended further in \cite{walz1987consequences}.
In these papers it was exhibited a chemical network yielding the adaptation property. This network satisfies the detailed balance property and does not exchange substances with the environment, hence it is closed. 
However the mechanism found in \cite{segel1986mechanism,walz1987consequences} requires a fine tuning of the parameters and, therefore, the results in these papers do not contradict the statement above concerning the impossibility of having adaptation in closed systems in a robust manner. 
 Hence there is no contradiction between the results obtained in our paper and the results in \cite{segel1986mechanism, walz1987consequences}.  
 Notice that, as explained also in \cite{barkai1997robustness} the robustness of the adaptation property is required in order to guarantee the proper biological functioning of the signalling system.

 In this paper we are interested in analysing the relation between the detailed balance property and the property of adaptation. The detailed balance property is a property that fails for all the chemical networks that contain a reaction that is one directional, i.e. that contain at least one reaction of the form $A \rightarrow B $. 
 As a consequence in this paper we focus on the analysis of bidirectional chemical networks, i.e. of chemical networks that contain only reactions of the form $A\leftrightarrows B$.  
Notice, however, that many chemical networks in the literature contain only one directional reactions of the form $A \rightarrow B $. In Section \ref{sec:examples} we explain in a concrete example how these one directional reactions can be seen as limits of bidirectional reactions for which the detailed balance property holds. 
We stress that this problem has some similarities with the study in \cite{gorban2011extended}, where it is proven that one directional reactions can be obtained as limits of bidirectional reactions for which the detailed balance property holds.

We stress that, as explained in \cite{franco2025reduction}, the detailed balance is a property that, at the fundamental level, must be satisfied by any chemical networks.
In \cite{franco2025reduction} we have studied the properties of chemical networks that are obtained by freezing the concentration of some of the substances in a chemical network that satisfies the detailed balance property. 
In particular in \cite{franco2025reduction} we prove that these chemical networks, in which a set of substances are kept at constant values by the exchange of chemicals with the environment, do not satisfy the detailed balance, unless certain topological conditions are satisfied or the frozen concentrations are chosen at equilibrium values.

In this paper we prove that robust adaptation is impossible in closed systems, i.e in systems that satisfy the detailed balance property and do not exchange substances with the environment.
Notice also that this is coherent with the results in the literature, where models of robust adaptation do not satisfy the detailed balance property. 
It is therefore natural to ask the question whether the models of adaptation are admissible and can be obtained by freezing the concentration of certain substances in a chemical network that satisfies the detailed balance property. 
In Section \ref{sec:examples} we will show that this is the case for some models of adaptation that can be found in the literature.

\begin{figure}[H] 
\centering
\includegraphics[width=0.8\linewidth]{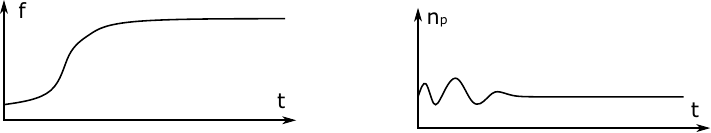}
\caption{
On the left we plot the function $f$, describing the change in time of the concentration of signal. 
The function $f$ tends to a constant value as $t\rightarrow \infty $. 
On the right we plot $n_p(t)$, which is the concentration of product, for a signalling system that satisfies the adaptation property. Notice that the concentration of product changes as the signal changes. However it returns to the pre-signal values as time tends to infinity. 
}\label{fig1}
\end{figure}

 \bigskip 

\textbf{Main results} 
 \bigskip

 For the convenience of the reader we state here the three main results of this paper in an informal way. 
The goal of this paper is to study the relation between the property of adaptation and the property of detailed balance. 
The main result that we prove is the fact that robust adaptation does not take place in closed systems unless the conserved quantities satisfy a very specific assumption that we refer to as \textit{factorization assumption}. In this paper we say that the adaptation property is robust (or stable) if the property remains true under small changes of the reaction rates. 
As a consequence, when the adaptation property is not robust, the chemical rates are fine tuned.  

Before stating the main results of this paper let us state in an informal way the three properties that are required in order to have adaptation. 
A signalling system satisfies the adaptation property if the following three conditions are satisfied 
\begin{enumerate}
    \item the vector of the concentrations of the substances in the system converges to a steady state as time tends to infinity; 
    \item changes in the signal concentrations produce changes in the concentration of the product (i.e. if the function $f$ is not constant, then we have that the concentration of product $n_p$ is also not constant); 
    \item after a transient time, the concentration of product returns to the pre-signal values, i.e. to the values that it had at time $t=0$, which is $n_0^T e_p$. 
\end{enumerate}
The most important result of this paper is the following. 
       \begin{theorem}[Stable adaptation is impossible in closed systems] \label{thm:intro ad in closed}
       Assume that a signalling system is closed. 
    Then the signalling system does not satisfy the adaptation property unless the reaction rates are fine tuned or the conserved quantities satisfy a factorization assumption. 
    \end{theorem}
The rigorous statement of the above result is Theorem \ref{thm:no adaptation in closed systems}. The factorization assumption mentioned in the theorem above, that will be defined later in Section \ref{sec:adaptation}, implies a very particular structure of conserved quantities that might not be admissible in concrete models.
The factorization assumption states that we can write every conservation law as the weighted sum of $k$ conservation laws $ m_i=(m_{j_i})$ for $i =1 \dots k $, that
divide the set of the substances  $\Omega $ in disjoint subsets, i.e. $ \Omega = \cup_{i =1}^k \Omega_i$.
More precisely these conservation laws are of the following form
\[
\sum_{j_i \in \Omega_i} m_{j_i } n_{j_i}(t)=\sum_{j_i \in \Omega_i} m_{j_i } n_{j_i}(0) , \quad \forall t \geq 0,  \quad  \forall i \in \{1, \dots, k \}
\]
where the vector $n$ is the solution to the system of ODEs \eqref{eq:intro signalling}.

Moreover in this paper we also find a mechanisms yielding adaptation in signalling systems whose underlying kinetic system satisfies the detailed balance property. 
Indeed, the first property required in order to have adaptation follows by the fact that when the kinetic system underlying the signalling system satisfies the detailed balance property and when the function $f$, describing the change in time of the signal concentration, converges to a constant value sufficiently fast, then the vector of the concentrations $n(t)$ of the signalling system converges to a steady state. This will be explained in detail in Section \ref{sec:signalling}. 
The second property required to have adaptation is that the system should respond to changes of the signal concentration, this is analysed in Section \ref{sec:response} when the kinetic system underlying the signalling system satisfies the detailed balance property. In particular we prove Theorem \ref{thm:reaction}, that we state here in an informal manner. 
       \begin{theorem}[Response in a connected system]\label{thm:intro response}
   Consider a connected signalling kinetic system. Assume that the underlying kinetic system satisfies the detailed balance property. Then changes in the signal concentration produce changes in the concentrations of all the substances in the system, unless the parameters are fine tuned. 
    \end{theorem}
It turns out that if a connected signalling system satisfies the detailed balance property (i.e. if the underlying kinetic system satisfies the detailed balance property) and exchanges one substance with the environment, then the third property required to have adaptation holds. This is the content of Proposition \ref{prop:DB and adaptation}. 
As a consequence we have the following statement, formulated informally here.
      \begin{theorem}[Adaptation in kinetic systems with the detailed balance property]\label{thm:intro db and adaptation}
      Consider a connected signalling kinetic system. Assume that the underlying kinetic system satisfies the detailed balance property and is not conservative (i.e. it exchanges at least one substance with the environment). Then, unless the parameters are fine tuned, the adaptation property holds. 
    \end{theorem}

\bigskip 

\textbf{Plan of the paper} 
\bigskip 

The plan for this paper is the following. 
In Section \ref{sec:signalling} we introduce the definition of signalling system. To this end we briefly recall the definition of kinetic system (see Section \ref{sec:kinetic systems}) and the definition of detailed balance for kinetic systems (see Section \ref{sec:db}). As a last step we introduce the definition of a signalling system as a kinetic system in which the signal changes in time according to a given function of time $f$. This is done in Section \ref{sec:signalling def}. Moreover, in this section we also prove that if the kinetic system underlying the signalling system satisfies the detailed balance property and if the function $f$, that prescribes the way in which the signal changes in time, converges sufficiently fast to a constant value, then the the vector of the concentrations of the substances in the system converges to a steady state as time tends to infinity. 

In Section \ref{sec:adaptation} we prove one of the most important results of this paper regarding the possibility of having adaptation in signalling systems that exchange chemicals with the environment. 
In particular, we rigorously state and prove Theorem \ref{thm:intro response}. 
As we will explain later, in order to prove Theorem \ref{thm:intro response} we assume that the kinetic system, underlying the signalling system under study, satisfies detailed balance property. The reason why we make this assumption is that kinetic systems with the detailed balance property have a structure that allows to prove that changes in the signal produce changes in the product. It would be an interesting problem to study this problem in full generality, i.e. removing the detailed balance assumption.  
In Section \ref{sec:db and adaptation} we find a mechanism of having adaptation in a signalling system that satisfies the detailed balance property. 
In other words, we prove Theorem \ref{thm:intro db and adaptation}. 
Finally in Section \ref{sec:adaptation in closed systems} we prove the main theorem of the paper, i.e. we prove that, unless the factorization assumption on the conserved quantities holds, robust adaptation is impossible in closed systems. This is the statement of Theorem \ref{thm:intro ad in closed}. 

In Section \ref{sec:examples} we compare the results obtained in this paper with some models of adaptation that can be found in the literature. These include the Barkai-Leibler model of bacterial chemotaxis, the classical model of adaptation proposed in  \cite{segel1986mechanism} and some of the models considered in \cite{ferrell2016perfect}. 
In particular we show that our results are consistent with the results in the literature. 
Moreover, we also show that some models of adaptation that are not closed can be obtained as the reduction of systems that are closed. More precisely, they are effective models describing kinetic systems that are closed and in which the concentration of  certain substances is frozen at constant values.

\bigskip 

\subsection{Notation}
In this paper we use the following notation. We define $\mathbb R_+$ and $\mathbb R_*$ to be given respectively by $ \mathbb R_* = [0, \infty )$ and $ \mathbb R_+ = (0, \infty )$. 
Moreover we denote with $e_i \subset \mathbb R^n $, for $ i \in 1, \dots n $, the vectors of the canonical basis of $\mathbb R^n$. 
Given two vectors $v_1 , v_2 \in \mathbb R^n $ we denote with $\langle v_1, v_2 \rangle $ their euclidean scalar product in $\mathbb R^n$ and with $v_1 \otimes v_2 $ we denote their tensor product.
Finally, given a vector $v \in \mathbb R^n $, we will denote with $e^{v}$ the vector $ (e^{v(i)})_{i=1}^n \in \mathbb R^n$. Similarly it will be useful to denote with $\log(v) $ the vector $ (\log(v(i)))_{i=1}^n \in \mathbb R^n$. 
Given a $z \in \mathbb R^n$, we denote with $B_r (\zeta)$ the open ball of radius $r$  around $z$, i.e. 
\[
B_r(z)=\{ y \in \mathbb R^n : \| z-y \| < r   \} 
\]
where $\| \cdot  \| $ is the Euclidean norm. 

Let $\mathcal G =(V, E) $ be a graph. 
A walk w in $\mathcal G $ is a finite non-null sequence $v_0 e_1 v_1 e_2 v_2 \dots e_k v_k $ whose terms are alternatively vertices and edges such that, $e_i=(v_{i-1}, v_i)$ for every $ i \in \{ 1 \dots k \}$.  Moreover, $\ell(w):= k $. A path is a walk in which all the edges and all the vertices are distinct. In Section \ref{sec:response} we will use often the notation $i \in w $ where $w $ is a walk and $i \in V $ to indicate that the walk $w$ contains the vertex $i $. 
Moreover, unless not otherwise specified, we say that $f \sim g $ as $ t \to \infty$ (or as $t \to 0$) if $\lim_{t\to \infty } \frac{f(t) }{g(t) } =1$ (or $ \lim_{t \to 0} \frac{f(t) }{g(t)}=1$). 
 
\section{Signalling systems} \label{sec:signalling}
The goal of this section is to give the definition of signalling system. 
As mentioned in the introduction a signalling system is a kinetic system where one substance, the signal, changes in time according to a given external boundary condition $f$. Therefore we start this section defining a kinetic system.

\subsection{Kinetic systems and their conservation laws} \label{sec:kinetic systems}
In this section we recall the definition of kinetic system.
 A kinetic system is a set of of substances, a set of reactions and a set of reaction rate functions.
 \begin{definition}[Kinetic system]
  Let $\Omega:=\{ 1, \dots, N \}$. 
  Let $r \geq 1 $ and $\mathcal R := \{ R_1, \dots, R_r  \}  $ where $R_j \in \mathbb Z^N\setminus \{ 0\} $ for every $j \in \{ 1, \dots, r \} $. 
Let $\mathcal K : \mathcal R \rightarrow \mathbb R_+$ be a function. 
Then $(\Omega , \mathcal R, \mathcal K)$ is a  kinetic system.
\end{definition}
The set $\Omega $ is the set of the substances in the system, $\mathcal R$ is the set of chemical reactions taking place in the network and $\mathcal K (R) $ is the rate of the reaction $R \in \mathcal R$.  The function $\mathcal K $ is usually called reaction rate function. 

Let $R \in \mathcal R $ be a reaction. We will use the following notation
\[
I(R):=\{ i \in \Omega : R(i) <0 \}, \quad   F(R):= \{  i \in \Omega : R(i) >0 \} \ \text{ and } \ D(R):=I(R) \cup F(R).
\]
 The set $I(R)$ is the set of the initial substances of the reaction $R \in \mathcal R$ and the set $F(R) $ is the set of the final substances of the reaction $R \in \mathcal R$. The set $D(R) $ is the domain of the reaction, i.e. the set of the substances that take part to the reaction. 
In this paper we assume that $I(R) \cap F(R) = \emptyset $ for every $R \in \mathcal R$. This means that reactions of the form $(1) +(2) \rightarrow (2) +(3) $ are not considered in this paper. 
Moreover, we assume that every substance in the system takes part to at least one reaction, this means for every $i \in \Omega $ there exists a $R \in \mathcal R$ such that $ i \in  D(R)$ and that $D(R)\neq \emptyset$ for every $R \in \mathcal R$.

We give the definition of conservation law. 
\begin{definition}[Set of conservation laws]
    The set $\mathcal M $ of conservation laws of a chemical network $(\Omega , \mathcal R) $ is defined as
    \begin{equation} \label{eq:stochio}
\mathcal M  := \operatorname{span}\{ R: R \in \mathcal R  \}.  
    \end{equation}
\end{definition}

Let $m \in \mathcal M $, then
    \[
     m^T R = \textbf{0},\quad \forall R \in \mathcal R. 
    \]
    This is the reason why we refer to $\mathcal M $ as the set of the conservation laws. If $n_0 \in \mathbb R^N $ is the initial vector of concentrations $n_0$ and $n(t)$ is the vector of concentrations at time $t>0 $, then 
\[
 m^T n_0 = m^T n(t) \quad \text{ for every } t >0 \text{ and for any } m \in \mathcal M. 
\]
We define now the set of physically relevant non-negative conservation laws
$\mathcal M_+:=\mathcal M \cap \mathbb R_*^N .$
We say that the kinetic system $(\Omega, \mathcal R, \mathcal K )$ is conservative if
\begin{equation}\label{eq:conservative}
\mathcal M_+ \cap \mathbb R_+^N \neq \emptyset.
\end{equation}
It is possible to prove that it is always possible to find a positive basis of $\mathcal M $ when the kinetic system $(\Omega , \mathcal R, \mathcal K ) $ is conservative. 
\begin{lemma} \label{lem:positivity of the conservation laws}
Assume that the chemical network $(\Omega, \mathcal R) $ is conservative. 
Then the set of the extreme rays $\mathcal B$ of the positive cone $\mathcal M_+$ are a basis of $\mathcal M $. 
\end{lemma}
For the proof of Lemma \ref{lem:positivity of the conservation laws} we refer to \cite{franco2025reduction}.

Let us define the subset of reactions $\mathcal R_s \subset \mathcal R $, obtained identifying each reaction $R$ with the reversed reaction $- R$. 
More precisely, the set $\mathcal R_s \subset \mathcal R$ is defined as 
\[
\mathcal R_s := \{ R \in \mathcal R : - R \notin \mathcal R  \} \cup  \{ R \in \mathcal R \setminus \{ R \in \mathcal R : - R \notin \mathcal R  \} : \min{I(R) } < \min {F(R)} \}.
\] 
Hence we have that, if $R \in \mathcal R $ is such that $-R \notin \mathcal R $, then $R \in \mathcal R_s $. Instead if $ R ,  - R \in \mathcal R$ only one of the two reactions belong to $\mathcal R_s$. 
We say that a kinetic system is \textit{bidirectional} if for every $R \in \mathcal R $ we have that $- R \in \mathcal R$.
Notice that if the kinetic system  $(\Omega, \mathcal R, \mathcal K )$ is bidirectional, then $|\mathcal R_s| =r/2$.

It is convenient to associate to a kinetic system $(\Omega , \mathcal R, \mathcal K ) $ a system of ODEs as follows
\begin{equation} \label{ODEs}
\frac{d n (t) }{dt} = \sum_{ R \in \mathcal R } K_R R  \prod_{i\in I(R)} {(n_i)}^{-R(i)} , \quad n(0)=n_0 \in \mathbb R_*^N.
\end{equation}
Here the solution $ n:=(n_1, \dots, n_N)^T \in \mathbb R_*^N $ describes the change in time  of the concentrations of substances in the network.  
For the purposes of this paper is convenient to rewrite the system of ODEs corresponding to a bidirectional kinetic system as follows
\begin{equation} \label{ODEs fluxes}
\frac{dn(t)}{dt} =  \sum_{ R \in \mathcal R_s } R J_R(n), \quad n(0)= n_0 \in \mathbb R_*^N,
\end{equation}
where \begin{equation} \label{eq:fluxes}
    J_R (n):= K_R \prod_{ j \in I(R)  } {(n_j)}^{-R(j)} - K_{- R }  \prod_{ j \in F(R)  } {n_j}^{R(j)}. 
\end{equation}
We refer to Lemma 2.10 in \cite{franco2025reduction} for the details of this computation.

\subsection{Detailed balance property of kinetic systems}
\label{sec:db}
A bidirectional kinetic system satisfies the detailed balance property if 
at the steady state each reaction is balanced by its reverse reaction. 
We state now the precise definition. 
\begin{definition}[Detailed balance property]
  A bidirectional kinetic system $(\Omega, \mathcal R, \mathcal K ) $ satisfies the detailed balance property if there exists a $\overline N \in \mathbb R_+^N $ of \eqref{ODEs} such that  
  \begin{equation}\label{eq:DB}
K_R \prod_{i\in I(R)}  {(\overline N_i)}^{-R(i) } = K_{-R} \prod_{i \in F(R) } {\overline N_i}^{R(i) } \  \text{ for all } R\in \mathcal R_s. 
  \end{equation}
\end{definition}
By the definition of detailed balance we have that $\overline N$ is such that $J_R(\overline N) =0$ for every $R \in \mathcal R$. In particular we have that $\overline N$ is a steady state of the system of ODEs \eqref{ODEs}. 
A consequence of the detailed balance property is that the steady states of \eqref{ODEs} can be written as a function of the vector $E \in \mathbb R^N$, where $E$ is the vector of the energies associated with each substance in the network.
\begin{lemma} \label{lem:db special ss}
Assume that the kinetic system $(\Omega, \mathcal R, \mathcal K)$ satisfies the detailed balance property. 
    Then there exists a vector $E \in \mathbb R^N$ such that the equality
    \begin{equation} \label{eq:energy}
    \mathcal E(R):= \log\left( \frac{\mathcal K(-R) }{\mathcal K(R) } \right) = \sum_{i \in \Omega} R(i) E(i) 
    \end{equation}
    holds for every $R \in \mathcal R_s$.
Moreover $N^* =(N^*_i)_{i=1}^N$ is a steady state of the system of ODEs \eqref{ODEs} if and only if
      \begin{equation}      \label{stst when DB}
      N^*_i=  e^{- E(i)}, \quad i \in \{ 1, \dots, N\} 
      \end{equation}
      where $E \in \mathbb R^N $  is a solution of \eqref{eq:energy}.
      \end{lemma}
The proof of this lemma can be found in \cite[Lemma 3.6]{franco2025reduction}. 
Notice that a consequence of this lemma is that if the kinetic system $(\Omega, \mathcal R, \mathcal K ) $ satisfies the detailed balance property, then equality \eqref{eq:DB} is attained at every steady state of the system of ODEs \eqref{ODEs}.

In the following, we refer to a solution to \eqref{eq:energy} as an \textit{energy} of the kinetic system.
Notice that the solution to \eqref{eq:energy} is not unique unless $\mathcal M = \{0\} $. 
Indeed if $E $ is a solution to \eqref{eq:energy} and $m \in \mathcal M $, then also $m+ E $ is a solution to \eqref{eq:energy}. Indeed for every $R \in \mathcal R $ it holds that
\[
\sum_{i \in \Omega} R(i) E(i) + \sum_{i \in \Omega} R(i)m (i) = \sum_{i \in \Omega} R(i) E(i) = \mathcal E(R). 
\]

We conclude with the definition of closed kinetic system, which is a system that does not exchange substances with the environment.
\begin{definition}
    A kinetic system $(\Omega, \mathcal R, \mathcal K )$ is closed if it satisfies the detailed balance property, is conservative and is such that 
    \begin{equation} \label{no sources/sinks}
   I(R) \neq \emptyset \ \text{ and } \  F(R) \neq \emptyset \text{ for every } R \in \mathcal R. 
    \end{equation}
\end{definition}
\subsection{Signalling systems} \label{sec:signalling def}
 In this section we give the definition of signalling systems. These are kinetic systems in which one of the concentration, the signal, changes in time according to a given function $f$. See Figure \ref{fig2} for a visual representation of a signalling system.
 More precisely, 
a \textit{signalling system}  $(\Omega , \mathcal R, \mathcal K, n_1(t))$ is a kinetic system in which the concentration of the substance $(1)$ is a given function of time. 
 Therefore, the change in time of the concentrations of the substances of a signalling system is described by the following system of ODEs 
\begin{equation} \label{ODEs signalling}
 \frac{d n (t) }{dt } = \sum_{ R \in \mathcal R_s } R J_R(n) + J^F, n_0 \in \mathbb R_*^N
\end{equation}
 where
$ J^F(t) := e_1 J^F_1 (t)$ and where $J^F_1 (t)$ is given and is such that $n_1(t)=f(t) $. 
Moreover, we make the following assumptions on the function $f(t)$. This assumption guarantees that the signal concentration converges to constant values sufficiently fast. 
\begin{figure}[H] 
\centering
\includegraphics[width=0.5\linewidth]{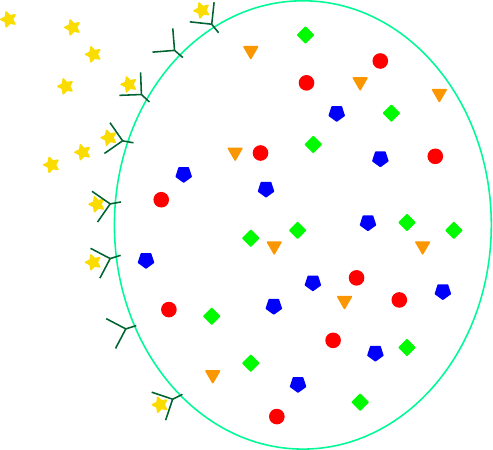}
\caption{
An example of signalling system that  is conservative, i.e. it exchanges only the signal (yellow star) with the environment. Notice that the signals binds to the receptors of the signalling system. 
}\label{fig2}
\end{figure}

\begin{assumption}[Assumption on the signal] \label{ass signal}
We assume that $f: \mathbb R_* \rightarrow \mathbb R_* $ is a continuously differentiable function such that there exists a $r >0 $ such that $|\frac{d}{dt}\log (f)|< e^{ - rt } $ and
\begin{equation} \label{singal approches a constant}
\lim_{t \rightarrow \infty } f(t) = \overline n_1>0. 
\end{equation}   
\end{assumption}
\begin{proposition} \label{prop:property 1 adaptation}
    Assume that $(\Omega, \mathcal R, \mathcal K) $ satisfies the detailed balance property and is such that $\mathcal M \neq \{ 0\} $.  Let $n_0 \in \mathbb R_*^N $. 
Assume that the function $f(t)$ satisfies Assumption \ref{ass signal} and is such that $f(0)=n_0(1)$. 
Then the solution of the system of ODEs \eqref{ODEs signalling} is such that
\begin{equation} \label{bound fluxes}
\overline J^F_1:=\int_0^\infty J_1^F (t) dt < \infty.
\end{equation}
Moreover, for every $n_0\in \mathbb R^N $ we have that 
$\lim_{t \to \infty } n(t)= e^{-E}$
where $E$ is the unique energy of $(\Omega, \mathcal R, \mathcal K) $ that is such that
\begin{equation} \label{cons laws when fluxes2}
m^T e^{- E }- m^T n_0 = m^T \overline J^F_1 \text{ for every } m \in \mathcal M. 
\end{equation}

\end{proposition}
\begin{proof}
Notice that there exists an energy $E_0 $ of the kinetic system $(\Omega, \mathcal R, \mathcal K ) $ that is such that $E_0(1)= - \log(n_0(1))= - \log(f(0))$. 
Indeed since the kinetic system satisfies the detailed balance property we know that there exists an $\overline E $ such that \eqref{eq:energy} holds for every $R \in \mathcal R_s$. 
Assume that $ \overline E(1) \neq - \log(n_0(1))= - \log(f(0))$. 
Without loss of generality we assume that there exists a $m \in \mathcal M $ that is such that $m(1)>0$.
Then we can define $E_0$ as 
\[
E_0=\overline E - m  \frac{ \overline E+  \log(n_0) }{m(1)}
\]
where $m \in \mathcal M $ is such that $m(1) >0$. 

Let us fix a conservation law $m \in \mathcal M $ such that $m(1)>0$. Then we can define the function $E^s : \mathbb R_* \rightarrow \mathbb R^N  $ as 
\[
E^s(t):=E_0 + \frac{m }{m(1) } \left[\log(n_0(1)) - \log(f(t)) \right]. 
\]
We can then define the function $n^s$ as $n^s(t):= e^{- E^s(t)}.$
Notice that by the definition of $E^s$ we have that
\[
\frac{d }{dt}  n_j^s(t)= - n_j^s(t) \frac{d }{dt}  E_j^s(t)= \frac{m(j) }{m(1)} n_j^s(t) \frac{1}{f(t)}  \frac{d f(t)}{dt}, \quad  \forall j \in \{1, \dots, N\}. 
\]
Moreover, notice that since $f$ is assumed to be bounded, then we also have that $n^s$ is bounded. 
Let us consider the function $F_r: \mathbb R_*^N \rightarrow \mathbb R $ defined as 
\[
F_r(n) := \sum_{\Omega \setminus \{ 1\} } n_j \left( \log \left(\frac{n_j}{n_j^s} \right) -1  \right). 
\]
We have that 
\begin{align*}
 \partial_t  F_r(n) = \sum_{\Omega \setminus \{ 1\} } \partial_t n_j \log \left(\frac{n_j}{n_j^s} \right)  - \sum_{\Omega \setminus \{ 1\} } \frac{n_j}{n_j^s}\frac{d}{dt} n^s_j  = \mathcal D_r (n)   - \sum_{\Omega \setminus \{ 1\} } \frac{m(j)}{m(1)}n_j (t) \frac{1}{f(t)}  f'(t)
\end{align*}
where we have that 
\[
\mathcal D_r (n) := \sum_{ R \in \mathcal R } K_R \prod_{ i \in \Omega \setminus \{1\} } n_i^{- R(i) }  \log \left(\prod_{j \in \Omega \setminus \{ 1 \} } \left( \frac{n_j}{n^s_j} \right)^{R(j) }  \right) \left( 1- \prod_{k \in \Omega \setminus \{1 \} } \left(\frac{n_k}{n^s_k} \right)^{R(k) }  \right) \leq 0. 
\]
As a consequence, using the fact that $f$, $f'$ and $n_s $ are bounded we notice that for every $t>0$
\[ 
 \partial_t  F_r(n) \leq - \sum_{\Omega \setminus \{ 1\} }\frac{m(j)}{m(1)}n_j(t)\frac{1}{f(t)}  f'(t) < c_1 e^{-rt} \sum_{\Omega \setminus \{ 1\} }n_j(t) \leq c_1  e^{-rt} \left(c_2 + F_r(n) \right)
\]
for some constants $c_1, c_2 >0 $. 
By Grönwall's inequality we deduce that for every $t >0 $ we have that $F_r(t) < \infty$.
Moreover, we deduce also that $\sup_{ t >0 }  F_r(n) < \infty $, hence $ \sup_{ t >0 } n (t)  < \infty $. 
On the other notice that since there exists a $m \in \mathcal M $ such that $m(j) >0 $ for every $j \in \Omega $ we deduce that
\[ 
m^T n (t)  - m^T n_0 = m(1) \int_{0}^t J_1^F (s) ds  
\]
taking the supremum as $ t \to \infty $ we deduce that $\int_{0}^\infty  J_1^F (s) ds < \infty $. Moreover we deduce that $\lim_{ t \to \infty} n (t)=n(\infty) $ where $n(\infty) \in \mathbb R_*^N$.  
From this we deduce that since $ \int_0^\infty \mathcal D_r(n(s) ) ds < \infty  $ and $\lim_{t \to \infty } \mathcal D_r(n) < \infty $, then $ \lim_{ t \to \infty} \mathcal D_r(n)=0$. By continuity this implies that $\lim_{ t \to \infty} n(t) = \lim_{ t \to \infty} n^s(t)= e^{- E}$
where $E$ is the unique energy of $(\Omega , \mathcal R, \mathcal K ) $ such that $E_1 =-  \log( \overline n_1 )$ that satisfies \eqref{cons laws when fluxes2}. 
\end{proof}

\section{Adaptation} \label{sec:adaptation}
In this section we study the property of adaptation for signalling systems. 
The property of adaptation is an important property of several biological systems. Since these system often operate in out of equilibrium conditions it is relevant to understand if the adaptation property can be satisfied only by systems that are out of equilibrium. 
In particular, we are interested in understanding if signalling systems with the property of detailed balance can or cannot have the property of adaptation. 

\subsection{Definition of adaptation}\label{sec:adaptation def}

In this section we study the property of adaptation for signalling systems $(\Omega , \mathcal R, \mathcal K, n_1(t)) $.
A signalling system satisfies the property of adaptation if there exists a substance $p \in \Omega $ whose concentration changes when the signal concentration changes, but returns to the initial concentration levels (i.e. $n_0^T e_p$) after this transient time.
The precise definition of adaptation that we will discuss in this paper is the following one. 
\begin{definition}[Adaptation] \label{def:adaptation}
Consider the signalling system $(\Omega, \mathcal R, \mathcal K , n_1(t)) $ where $n_1(t)=f(t)$ and $f$ satisfies Assumption \ref{ass signal}. Assume that $n$ is the solution of equation  \eqref{ODEs signalling} with initial datum $n_0=e^{- E} \in \mathbb R_*^{N} $ where $E$ is an energy of $(\Omega, \mathcal R, \mathcal K) $. 
We say that the signalling system 
satisfies the adaptation property with respect to the signal $(1)$ and the product $ p \in \Omega \setminus \{ 1 \} $ if it satisfies the following three conditions.
\begin{enumerate}
\item Let $ \lim_ {t \to \infty} f(t) = \overline n_1 $. Then, there exists a unique steady state $N[ \overline n_1] \in \mathbb R_+^N$, depending on $\overline n_1 $, such that 
\[
\lim_{t \to \infty } n(t) = N[ \overline n_1];  
\]
    \item  the steady state $N[\overline n_1]$ satisfies $N[ \overline n_1](p)=n_0^T e_p$;
\item it holds that $ \sup_{ t >0} | n_p (t ) - n_0^T e_p | > 0 . $
\end{enumerate}
\end{definition}

The first property guarantees that the system of ODEs \eqref{ODEs signalling} converges to a steady state as time tends to infinity. The second property guarantees that the concentration of the substance $(p)$ at this steady state does not depend on the limiting value of the signal $ \overline n_1 $. In particular the concentration of the substance $(p)$ is given by the initial concentration of substance $(p)$, i.e. $n_0^T e_p$.
This guarantees that if we assume that the system is initially at steady state and we perturb it by changing the concentration of the signal, then, as time tends to infinity, the concentration of $p$ converges to the steady state value that it had before the change in the signal. 
The last property guarantees that the system reacts to changes of the signal.
We refer to Figure \ref{fig1} for a graphical representation of the property of adaptation. 

In the next section we study topological conditions on the chemical network $(\Omega, \mathcal R) $ that guarantee that property 3 in Definition \ref{def:adaptation} holds when the kinetic system satisfies the detailed balance condition. 
\subsection{Response of the network to changes in the signal concentration} \label{sec:response}
For the purposes of this section it is convenient to consider the following graph associated with a kinetic system $(\Omega, \mathcal R ) $.
The graph $\mathcal G_\Omega=(V, E) $ has vertices 
\[
V:= \Omega 
\]
and edges 
\[
E:=\{(\alpha , \beta  ) \in V : \exists R \in \mathcal R \text{ with } R(\alpha ) R(\beta ) \neq 0 \}. 
\] 
Two substances are \textit{connected} if there exists a walk in $\mathcal G_\Omega $ that connects them.

The following example shows that the fact that the substance $1\in \Omega $ is connected with $p \in \Omega \setminus \{ 1\} $ is not a sufficient condition for the property 3 in Definition \ref{def:adaptation} to hold. 
\begin{example}[Connected system with detailed balance and with no response] \label{ex1:nonresponse}
    Consider the following set of reactions 
    \[
    (1) \leftrightarrows (2),\quad  (1) \leftrightarrows (3), \quad  (2)\leftrightarrows (3) + (4). 
    \]
    We assume that all the reactions take place at rate $1$ except for the reactions  $ (2) \rightarrow (1) $ and $(3) \rightarrow (1) $ that take place at the same rate $\alpha >0$. 
    Notice that we can define a kinetic system $(\Omega, \mathcal R, \mathcal K ) $ corresponding to this set of chemical reactions. Moreover, we stress that the substance $(1)$ is connected with the substance $(4)$. 
    Assume that the concentration of $1 $ is given i.e. $n_1(t)=f(t)$ for a function $f$ that satisfies Assumption \ref{ass signal}. Then the evolution of $(n_2,n_3, n_4) $ is given by 
   \begin{align*}
       \frac{d n_2 }{dt } &=f(t)- (\alpha +1) n_2 + n_3 n_4  \\ 
         \frac{d n_3 }{dt } &=f(t) - \alpha  n_3 + n_2 - n_3 n_4 \\
           \frac{d n_4 }{dt } &= n_2 - n_3 n_4. 
   \end{align*}
   We consider the initial condition $n_0 \in \mathbb R_*^3$ given by $n_0=(\frac{f(0) }{\alpha},\frac{f(0) }{\alpha},1)^T$. 
   We stress that the above system of ODEs defines a signalling system. 
   Notice that a steady state of the above system must be of the form $ (N_1, N_1, 1)$. 
Moreover, 
   \begin{align*}
       \frac{d \left( n_2 - n_4 n_3 \right)  }{dt }=\frac{d n_2 }{dt }  - n_3 \frac{d  n_4 }{dt }  - n_4 \frac{d   n_3  }{dt } = F(n_2, n_3, n_4) 
   \end{align*}
   where \[
  F(n_2, n_3, n_4) : = (1- n_4) f(t) - (1+\alpha) (n_2 - n_4 n_3) - (n_4+n_3 ) (n_2 - n_4n_3  ).  
   \]
   Now notice that for every $N_1 >0 $ we have that $F(N_1, N_1,1)=0$.

 We now assume that $f(t) \rightarrow \overline n_1 \neq f(0)$ as $t \to \infty $. Hence we assume that the signal changes in time.  
 However $n_2 (t) - n_4(t)  n_3 (t) =0$ for every $t>0 $ because $F(n_2, n_3, n_4) =0$ for every $t >0$. Hence $n_4(t)=1 $ for every $t >0 $. As a consequence there is no change in the concentration of the substance $(4)$ after changes of the signal concentration. Therefore, the property 3 in the definition of adaptation does not hold even if the network is connected. 
\end{example}

In Example \ref{ex1:nonresponse} the product and the signal are connected , however changes in the signal do not produce changes in the product concentration.
However, the parameters in Example \ref{ex1:nonresponse} are fine tuned. 
Indeed, if we perturb the rates of the reactions taking place in the network, then we obtain that the property $3$ in Definition \ref{def:adaptation} holds. This motivates us to prove the following theorem that states that if the kinetic system $(\Omega, \mathcal R, \mathcal K) $ satisfies the detailed balance condition and is connected, then the property 3 in Definition \ref{def:adaptation} can fail only in an unstable manner, i.e. it fails only for fine tuned reaction rates. 
Before stating the theorem we present an additional assumption that the function $f$, describing the change in time of the signal, should satisfy in order for the theorem to be valid. In particular the signal concentration $n_s=f $ must change in the following way for small times
\begin{equation} \label{linear change signal}
f (t) - n_0(1) \sim  t \text{ as } t \rightarrow 0^+.
\end{equation}
A function $f$ that satisfies both \eqref{linear change signal} and Assumption \ref{ass signal} is called \textit{admissible signal}. 
\begin{theorem} \label{thm:reaction}
Assume that $(\Omega, \mathcal R, \mathcal K) $ satisfies the detailed balance property. 
Assume that $1 , p \in \Omega $ are connected. Then for every $\delta >0$ there exists a rate function $ \overline {\mathcal K} : \mathcal R \rightarrow \mathbb R_+$ such that
\begin{enumerate} 
\item $\| \mathcal K - \overline {\mathcal K} \| < \delta $; 
\item $(\Omega, \mathcal R, \overline {\mathcal K})$ satisfies the detailed balance property; 
\item the solution $\overline n $ to the system of ODEs \eqref{ODEs signalling} induced by $(\Omega , \mathcal R , \overline{\mathcal K} ) $ for an admissible signal satisfies property $3$ in Definition \ref{def:adaptation}. 
\end{enumerate}
\end{theorem}

\begin{proof}
The idea of the proof is the following. 
Initially the system is at steady state. We consider a small perturbation of the substance $(1)$ and study the evolution of the concentration of $(p) $ for small times. Since we are considering a small perturbation and we are considering small times we can study the linearization of the system of ODEs \eqref{ODEs signalling}. 
As a second step we define a hierarchy on the elements $\Omega \setminus \{ 1\} $ that are connected to the substance $(1)$. In particular the hierarchy is based on the distance of the elements from $1$. The distance here has to be understood as distance in the graph $\mathcal G_\Omega $. 
The reason why we introduce this hierarchy is that the concentration of substances that are closer to $1$ change faster than the concentration of substances that are far from the signal $(1)$. We will make this idea rigorous in Step 2. Once the hierarchy is 
defined we can prove that, unless the parameters are fine tuned, we have a response in each point in $\Omega \setminus \{1\} $ that is connected with the substance $(1)$. This is done in Step 3. 

\textbf{Step 1: Linearization.} \\
Since the kinetic system satisfies the detailed balance property we have that $n_0=e^{-E} $ where $E$ is an energy of $(\Omega , \mathcal R, \mathcal K)$, i.e. $E$ is a solution to \eqref{eq:energy}. 
We can write the solution $n$ to the signalling system \eqref{ODEs signalling} as $n= e^{- E} (1 + \overline\varphi ) $ where, since by assumption $n_1 - n_1(0) \sim e^{E(1)} t $ as $t \rightarrow 0$, we have that  $\overline\varphi_1(t) \sim e^{E(1)} t $ as $ t \to 0 $. Moreover by the continuity in time of $n$ we have that $\|\overline \varphi (t) \| \to 0 $ as $ t \to 0$. 
Notice that since the detailed balance property holds we deduce that for every reaction $R \in \mathcal R $ it holds that   
\[
\log \left( \frac{K_{-R}}{K_R} \right) =\mathcal E(R)= \sum_{i \in \Omega} R(i) E(i). 
\]
Substituting this in \eqref{ODEs fluxes} we deduce that 
  \begin{equation} \label{flux DB}
    J_R(n) =K_R  \prod_{i \in I(R)} (n_i)^{- R(i) }  \left(1-  \prod_{i\in \Omega} n_i^{ R(i) } e^{ R(i) E(i) } \right), \forall n \in \mathbb R_*^N.
    \end{equation} 
Using the fact that $n = e^{- E} ( 1 + \overline \varphi) $ we obtain that
\begin{align*}
J_R(n)&=K_R \prod_{i \in I(R)} e^{ E(i) R(i)} (1+\overline\varphi_i)^{- R(i)} \left( 1- \prod_{k \in \Omega} (1+\overline\varphi_k{R(k)})\right) + o(\|\overline\varphi \|) \\
 &=- K_R \prod_{i \in I(R)} e^{ E(i) R(i)} (1-\overline \varphi_i R(i) )  \prod_{k \in \Omega} \overline \varphi_k{R(k)} + o(\|\overline \varphi \|) \\
 &= - K_R e^{ \sum_{i \in I(R)} E(i) R(i)}  \sum_{k \in \Omega} \overline \varphi_k{R(k)} + o(\|\overline \varphi \|). 
\end{align*}
Therefore equation \eqref{eq:fluxes} implies that
\begin{align*}
   e^{-E(\ell) } \frac{d  \overline\varphi_\ell }{ dt } = \sum_{R \in \mathcal R_s } R(\ell) J_R(n ) =-  \sum_{R \in \mathcal R_s } K_R \prod_{i \in I(R)}e^{  E(i) R(i)}   R (\ell)\sum_{k \in \Omega}\overline\varphi_k{R(k)} + o(\|\overline \varphi \|).
\end{align*}
Notice that to prove that $\sup_{ t >0} |\overline \varphi_p (t) |>0$ it is enough to show that $\sup_{t \in (0,T]} |\varphi_p(t) |>0 $ for every $T>0$, where $\varphi$ is the solution to the following linearized problem 
\begin{align} \label{eq:linearized odes}
\begin{cases}\varphi_1 (t) &= e^{ E(1)} t \\
\frac{d \varphi_\ell }{ dt } &=  \sum_{j \in \Omega } A_{\ell j } \varphi_j  \quad \ell \in  \Omega \setminus \{ 1 \} 
\end{cases} \end{align}
where 
\[ 
A_{\ell j } = - e^{E(\ell ) }  \sum_{ R \in \mathcal R_s}  R(j) R(\ell) K_R \alpha_R \  \text{ for } \ell \in \Omega \setminus \{ 1 \} \text{ and } j \in \Omega
\]
and where we are using the notation 
\begin{equation} \label{alphaR}
\alpha_R= \prod_{ i \in I(R) } e^{E(i) R(i)}.
\end{equation}

\textbf{Step 2: Hierarchy of responses.} \\
In order to prove that for every $T>0 $ we have that $\sup_{t \in (0, T] } |\varphi_p(t) | >0$ we construct a hierarchy of elements of $\Omega $ based on their distance from the signal $1 $. To this end we define the sets $S_0:=\{1\}$
and 
\[ 
S_n:= \{  i \in \Omega \setminus S_{n-1}: \exists R \in \mathcal R \text{ and } k \in S_{n-1} \text{ s.t. } R(i) R(k) \neq 0  \} 
\]
for every $n \in \{1, \dots , L-1  \}$ where $L $ is the minimal length of the walks connecting $1$ with $p$ in the graph $\mathcal G_\Omega $.
Moreover we assume that $S_L:=\{ p \} $. 

We are going to prove that  for every $\ell \in S_{n}$, with $n \in \{ 1, \dots, L \}$, we have that
\begin{equation} \label{eq:higer order terms linearized}
\varphi_\ell (t)= c_{ n +1, \ell  }(\mathcal K, E)  t^{n+1} + o(t^{n+1})
\end{equation}
where the constant $c_{ n } (\mathcal K ,  E) \geq 0$ is given by 
\begin{align} \label{eq:c}
 c_{ n, \ell  }(\mathcal K, E ):=& e^{E(1)} \frac{1}{n!}  \sum_{j_0 \in S_0} \sum_{ j_{1} \in S_{1} } A_{j_1 j_0} \sum_{ j_{2} \in S_{2} } A_{j_2 j_1 } \dots \sum_{ j_{n-2} \in S_{n-2} } A_{j_{n-2} j_{n-3} } \sum_{ j_{n-1} \in S_{n-1} } A_{j_{n-1} j_{n-2} } A_{ \ell j_{n-1}} \\
 =&e^{E(1)} \frac{1}{n!} \sum_{j \in X_n} a(j) \nonumber 
\end{align}
where $X_n = S_0 \times S_1 \times  \dots \times S_{n-1} \times \{ \ell \}  $, therefore $j =(j_0, j_1, \dots, j_{n-1}, j_n) $ with $j_0=1 $, $j_n=\ell $ and $j_i \in S_i $ for every $i \in \{ 1, \dots , n-1\} $. 
Moreover for $j \in X_n $ we have 
\[
a(j):= \prod_{i=1}^n A_{j_i, j_{i-1} }. 
\]
In other words $c_{n, \ell }  (\mathcal K , E)$ is defined inductively as follows.
We assume $c_{1, \ell } (\mathcal K , E) = e^{E(1)} A_{\ell 1 }$ for every $\ell \in S_1 $ and that $c_{n, \ell } (\mathcal K , E) = \sum_{j_{ n-1} \in S_{n-1}} c_{n-1, j_{n-1}}(\mathcal K , E) A_{\ell j_{n-1}}$ for every $n > 1 $ and $\ell \in S_n$. \\

We prove now \eqref{eq:higer order terms linearized}.
We start by proving the equality for $ n =1 $. More precisely we prove that 
\begin{equation} \label{varphi_ell1}
\varphi_\ell (t)= c_{1, \ell} (\mathcal K , E)  t^2+ o(t^2) \quad \forall \ell \in S_1.
\end{equation}
The system of equations \eqref{eq:linearized odes} implies that 
\[ 
\frac{d \varphi_\ell }{dt } = \sum_{j \in \Omega } A_{\ell j } \varphi_j = A_{\ell 1 } \varphi_1+ \sum_{j \in \Omega \setminus \{ 1 \}  } A_{\ell j } \varphi_j =t  e^{E(1)}  A_{\ell 1 } + \sum_{j \in \Omega \setminus \{ 1 \}  } A_{\ell j } \varphi_j. 
\]
Notice that by the definition of $c_{1, \ell } (\mathcal K , E) $ we have 
\[
 c_{1, \ell} (\mathcal K , E) = \frac{1}{2} e^{E(1) } A_{\ell 1 }, \quad \forall \ell \in S_1.
\]
As a consequence 
\[
\varphi_\ell (t)= c_{1, \ell}(\mathcal K , E) t^2 + \sum_{j \in \Omega \setminus \{ 1 \}  } A_{\ell j } \int_0^t \varphi_j(s) ds, \quad \forall \ell \in S_1.
\]
Notice that since for every $j \in \Omega \setminus (\{1\} \cup S_1)$ it holds that $\varphi_j (t)=o(t) $ as $t \to 0$ we have that \eqref{varphi_ell1} holds. If $L=1 $, then equality \eqref{eq:higer order terms linearized} follows for $\ell \in S_1$. 

Assume now that $L > 1$. 
We now prove \eqref{eq:higer order terms linearized} by induction. Therefore let us assume that for every $k \in S_{n-1} $ it holds that 
\[
\varphi_k (t)= c_{n-1, k}(\mathcal K , E) t^{n} + o(t^{n}). 
\] 
The equation for $\ell \in S_n $ is the following 
\begin{align*}
\frac{d \varphi_\ell }{ dt } = \sum_{ j \in S_{n -1} } A_{\ell j } \varphi_j + \sum_{i=n}^{L} \sum_{ j \in S_i } A_{\ell j } \varphi_j =   \sum_{ j \in S_{n -1} } A_{\ell j }  c_{n-1, j}(\mathcal K , E) t^{n} + o(t^{n}). 
\end{align*}
In the equality above we have used the fact that $\varphi_j (t) = o(t^n) $ as $t \to 0$ if $ j \in S_n$. 
Using the definition of $c_{n, \ell  } $ and integrating in time the equality above we obtain equality \eqref{eq:higer order terms linearized} for every $n \in \{ 0, \dots , L \} $.

\textbf{Step 3: There exists a rate function $\overline{ \mathcal K} $ as in the statement of the theorem and such that $c_{L,p} (\overline {\mathcal K}, \overline{E})  \neq 0 $.} \\
In order to prove step 3 it is useful to define the following set of reactions
\begin{equation} \label{gamma} 
    \Gamma_{n, n+1} := \left\{ R\in \mathcal R_s : \exists k \in S_{n+1}, \ell \in S_{n} \text{ s.t. } R(k) R(\ell) \neq 0 \right\} \text{ for every } n \in \{ 0, \dots, L \}. 
\end{equation}
We stress that if $n_1 \neq n_2 $, then $ \Gamma_{n_1, n_1+1} \cap  \Gamma_{n_2, n_2+1}  = \emptyset.$
Indeed, assume by contradiction that there exists $R \in \mathcal R_s $ such that $R \in \Gamma_{n_1,n_1+1} \cap \Gamma_{n_2,n_2+1} $ where $n_1\neq n_2 $, where without loss of generality we can assume that $n_1 < n_2$. This implies that there exist $k_1 \in S_{n_1} $ and $\ell_1 \in  S_{n_1+1} $ such that $R(k_1) R(\ell_1) \neq 0$ and there exist $k_2 \in S_{n_2} $ and $\ell_2 \in  S_{n_2+1} $ such that $R(k_2) R(\ell_2) \neq 0$. Notice that by the definition of the sets $S_n $ we have that since $\ell_2 \in S_{n_2+1} $, then $\ell_2 \notin S_{n_1+1} $. However notice that $R(\ell_2) R(k_1) \neq 0$, which implies that  $\ell_2 \notin S_{n_1+1} $.
This is a contradiction, hence the sets $\{ \Gamma_{n, n+1} \}_{n=0}^L$ are disjoint sets. 

The notation above allows us to rewrite the formula for some of the elements of the matrix $A$, i.e. for the elements $A_{\ell j } $ with $j \in S_n $ and $\ell \in S_{n+1} $. 
Indeed we have that 
\[
A_{\ell j } = - e^{E(\ell ) }  \sum_{ R \in \Gamma_{n, n+1} }  R(j) R(\ell) K_R \alpha_R \  \text{ for } \ell \in S_{n+1} \text{ and } j \in S_n \text{ and } 0 \leq n \leq L-1, 
\]
where we recall that $\alpha_R$ is given by \eqref{alphaR}.
We now substitute this formula for $A_{\ell j } $ in \eqref{eq:c}. 
In this way we obtain that 
\begin{align*}
c_{L, p } (\mathcal K , E)  &= 
%&= (-1)^L e^{E(1)} \frac{1}{L!} \sum_{ j_{1} \in S_{1} } e^{E(j_1)  }  \sum_{ R_1 \in \Gamma_{0, 1} }  R_1(j_1) R_1(1) K_{R_1} \alpha_{R_1}  \sum_{ j_{2} \in S_{2} }  e^{E(j_2)  }\sum_{ R_2 \in \Gamma_{ 1, 2 } }  R_2(j_1) R_2(j_2) K_{R_2} \alpha_{R_2}   \dots \\
%&\dots \sum_{ j_{L-2} \in S_{L-2} }  e^{E(j_{L-2})  }\sum_{ R_{L-2} \in \Gamma_{ j_{L-2}, j_{L-3} } }  R_{L-2} (j_{L-3}) R_{L-2}(j_{L-2}) K_{R_{L-2}} \alpha_{R_{L-2}}  \\ 
%& \sum_{ j_{L-1} \in S_{L-1} } e^{E(j_{L-1})  } \sum_{ R_{L-1} \in \Gamma_{ j_{L-1}, j_{L-2} } }  R_{L-1} (j_{L-2}) R_{L-1}(j_{L-2}) K_{R_{L-1}} \alpha_{R_{L-1}}   \\
%&   e^{E(p)} \sum_{ R_{L} \in \Gamma_{ j_{L}, j_{L-1} } }  R_{L} (p) R_{L}(j_{L-1}) K_{R_{L}} \alpha_{R_{L}} \\
 e^{E(1)} \frac{1}{L!} \sum_{j \in X_L}\prod_{i=1}^L A_{j_i j_{i-1}}= e^{E(1)} \frac{(-1)^L}{L!} \sum_{j \in X_L} \prod_{i=1}^L e^{E(j_i) } \sum_{R_i \in \Gamma_{i-1,i}} \alpha_{R_i} K_{R_i} R_i(j_i) R_i(j_{i-1}) = \\
&=(-1)^L  \frac{1}{L!} \sum_{ j \in X_L }   \sum_{ R_1 \in \Gamma_{0, 1} }  \sum_{ R_2 \in \Gamma_{ 1, 2 } } \dots   \sum_{ R_{L-2} \in \Gamma_{L-3,L-2} }    \sum_{ R_{L-1} \in \Gamma_{ L-2, L-1 } } \sum_{ R_{L} \in \Gamma_{ L-1, L } } e^{\sum_{i=0}^{L} E(j_i) }   \\
& \prod_{i=1}^L \alpha_{R_i} K_{R_i} R_i(j_{i-1}) R_1(j_i) 
=\sum_{j\in X_L} \sum_{r \in Y_L} v(r,j) 
\end{align*}
where $Y_L:=\Gamma_{0,1} \times \Gamma_{1,2} \times \dots \times \Gamma_{L-2,L-1} \times \Gamma_{L-1,L}$ and where $r =(R_1, R_2, \dots, R_{L-1}, R_L)$ with $R_i \in \Gamma_{i-1,i} $ for $i\in \{ 1, \dots, L\}$. Moreover given $j \in X_L $ and $r \in Y_L $ we have that 
\begin{align*}
v(r, j)&:= e^{\sum_{i=0}^{L} E(j_i) } \prod_{i=1}^L \alpha_{R_i} K_{R_i} R_i(j_{i-1}) R_i(j_i) =e^{\sum_{i=0}^{L} E(j_i) } \prod_{i=1}^L \alpha_{R_i} K_{R_i} R_i(j_{i-1}) R_i(j_i) \\
&= e^{\sum_{i=0}^{L} E(j_i) } \prod_{i=1}^L e^{\sum_{k \in I(R_i)} E(k) R_i(k)} K_{R_i} R_i(j_{i-1}) R_i(j_i) . 
\end{align*}
Notice that we have two possibilities, either $c_{L, p } (\mathcal K , E) \neq 0$ or $c_{L, p } (\mathcal K , E)= 0$.
If  $c_{L, p } (\mathcal K , E) \neq 0$ then we have that \eqref{eq:higer order terms linearized} implies that for every $T>0 $ it holds that $\sup_{ t \in [0, T] } |\varphi_p(t)|>0$. Hence the statement of the theorem follows. 

Assume instead that $c_{L, p } (\mathcal K , E) =0$. 
We want to construct a perturbation $\overline{\mathcal K}$ of the rate function $\mathcal K$ that is such that $c_{L, p }(\overline {\mathcal K}  , \overline E) \neq 0 $. 
Moreover, we want that the perturbed kinetic system $(\Omega, \mathcal R, \overline {\mathcal K} ) $ satisfies the detailed balance property. 
To construct this perturbation we argue as follows. We select the shortest path $\pi$ that connects $1 $ with $p$ in the graph $\mathcal G_\Omega$. Notice that this path $\pi$ identifies a sequence of reactions $\{ \overline  R_i\}_{i=1}^L $ and a sequence of points $\{ j_i^\pi\}_{i=0}^{L} $ in $\Omega$. Notice that by definition we have that $\overline R_i \in \Gamma_{i-1,i} $ for every $i \in \{ 1, \dots, L \} $. 
We modify the rates of the reactions $\{ \overline  R_i\}_{i=1}^L $. 
More precisely we consider the rate function $\overline {\mathcal K}: \mathcal R \rightarrow \mathbb R_+$ defined as 
\begin{align}
\overline {\mathcal K} (R) := \begin{cases} 
 \mathcal K(R) &\text{ if } R \notin \{ \overline R_i\}_{i=1}^L \text{ and } -R  \notin \{ \overline R_i\}_{i=1}^L  \\
 \mathcal K(R) + \delta_R &\text{ if } R \in \{ \overline R_i\}_{i=1}^L. 
\end{cases} 
\end{align}
Finally, we assume that 
\[ 
\overline{\mathcal K} ( -R) := \frac{\overline {\mathcal K} (R)  \mathcal K (-R)}{ \mathcal K (  R) } \exp\left( \sum_{i \in \pi } R(i) \varepsilon_i\right) \text{ if } R \in \{ \overline R_i\}_{i=1}^L
\]
for some $\varepsilon_i >0$, where $i \in \pi $. 
Notice that by construction the kinetic system $(\Omega, \mathcal R, \overline {\mathcal K} ) $ satisfies the detailed balance property. Indeed, by construction, we have that
\[
\overline {\mathcal E }(R) :=  \log(\overline {\mathcal K} (- R)/\overline {\mathcal K} (R)) = \log( {\mathcal K} ( - R)/ {\mathcal K} (R)) +   \sum_{i \in \pi } R(i) \varepsilon_i = \mathcal E (R)+   \sum_{i \in \pi } R(i) \varepsilon_i.
\] 
We recall that $\mathcal E$ is the vector of the energies of the reactions in the network and is defined as in \eqref{eq:energy} w.r.t. $(\Omega , \mathcal R, \mathcal K ) $. 
Hence \eqref{eq:energy} implies that 
\[ 
\overline {\mathcal E }(R)  =   \sum_{i \in \pi } R(i)E(i)+  \sum_{i \in \pi } R(i) \varepsilon_i
\] 
Therefore $\overline E = E+ \varepsilon $ is an energy of the perturbed kinetic system that, as a consequence, satisfies the detailed balance condition. 

The perturbed constant $ c_{L,p} (\overline {\mathcal K}, \overline E)  $ is given by 
\begin{align*}
 c_{L,p} (\overline {\mathcal K}, \overline E) &=\sum_{j\in X_L} \sum_{r \in Y_L}  e^{\sum_{k=0}^{L} \overline  E(j_k) } \prod_{i=1}^L e^{\sum_{k \in I(R_i)} \overline E(k) R_i(k)} \overline K_{R_i} R_i(j_{i-1}) R_i(j_i) \\
 &= \sum_{r \in Y_L}\prod_{\ell=1}^L e^{\sum_{k \in I(R_\ell)} \overline E(k) R_\ell(k)} \sum_{j\in X_L}   e^{\sum_{k=0}^{L} \overline  E(j_k) } \prod_{i=1}^L  \overline K_{R_i} R_i(j_{i-1}) R_i(j_i).   
\end{align*}

We now differentiate $ c_{L,p} (\overline {\mathcal K}, \overline E) $ with respect to the perturbations $\delta_{\overline R_k}$. We recall that we perturb on the reactions that define a path connecting $1$ with $p $ that has minimal length. Then we obtain that 
\begin{align*}
\left( \prod_{k=1}^L \frac{d}{d \delta_{\overline R_k} } \right)  c_{L,p} (\overline {\mathcal K}, \overline E) =& \sum_{r \in Y_L}\prod_{\ell=1}^L e^{\sum_{k \in I(R_\ell)} \overline E(k) R_\ell(k)} \sum_{j\in X_L}   e^{\sum_{k=0}^{L} \overline  E(j_k) }\prod_{i=1}^L R_i(j_{i-1}) R_i(j_i) \cdot  \\
& \cdot \left(  \prod_{k=1}^L \frac{d}{d \delta_{\overline R_k} }  \right) \prod_{i=1}^L  \overline K_{R_i} 
 \end{align*} 
Notice that 
\[
 \left(  \prod_{k=1}^L \frac{d}{d \delta_{\overline R_k} }  \right) \prod_{i=1}^L  \overline K_{R_i} \neq 0 \iff R_i=\overline R_ i \ \forall i \in \{ 1, \dots , L\}. 
\]
As a consequence we deduce that 
\begin{align*}
\left( \prod_{k=1}^L \frac{d}{d \delta_{\overline R_k} } \right)  c_{L,p} (\overline {\mathcal K}, \overline E) &= \prod_{\ell=1}^L e^{\sum_{k \in I(\overline R_\ell)} \overline E(k) \overline R_\ell(k)} \sum_{j\in X_L}   e^{\sum_{k=0}^{L} \overline  E(j_k) }\prod_{i=1}^L \overline R_i(j_{i-1}) \overline R_i(j_i). 
 \end{align*} 

 We use the notation
 \begin{align*} 
\Delta_\delta c_{L,p} (\overline {\mathcal K }, \overline E)  :=  \sum_{j\in X_L}   e^{\sum_{k=0}^{L} \overline  E(j_k) } \prod_{i=1}^L R_i(j_{i-1}) R_i(j_i). 
\end{align*}
Notice that if $\Delta_\delta c_{L,p} (\overline {\mathcal K } , \overline E )\neq 0$. Then the proof is finished. 
Indeed, this would imply that 
\[ 
\left( \prod_{k=1}^L \frac{d}{d \delta_{\overline R_k} } \right)  c_{L,p} (\overline {\mathcal K}, \overline E)
\neq 0 \] 
for sufficiently small values of $\delta $. Since for $\delta =0$ we have that $c_{L,p} (\overline {\mathcal K } , \overline E ) =c_{L,p} (\mathcal K  ,  E ) =0$
we have that for sufficiently small $\delta $ it holds that $ c_{L,p} (\overline {\mathcal K } , \overline E )\neq 0$. 

Therefore we assume now that $\Delta_\delta c_{L,p} (\overline {\mathcal K } , \overline E )= 0$.
We now differentiate $\Delta_\delta c_{L,p} (\overline {\mathcal K } , \overline E )$ with respect to the energy changes along the vertices of the path $\pi$, i.e. we consider 
\begin{align*}
\left( \prod_{i \in \pi } \frac{d}{d \varepsilon_i}    \right) \Delta_\delta c_{L,p} (\overline {\mathcal K}, \overline E) &= \sum_{j\in X_L} \left( \prod_{i \in \pi } \frac{d}{d \varepsilon_i} \right)  e^{\sum_{k=0}^{L} \overline  E(j_k) } \prod_{i=1}^L \overline R_i(j_{i-1}) \overline R_i(j_i)
 =  e^{\sum_{k=0}^{L} \overline  E(j^\pi_k) } \prod_{i=1}^L \overline R_i(j^\pi_{i-1}) \overline R_i(j^\pi_i) \\
 & \neq 0
 % &=\prod_{\ell=1}^L
 %e^{\sum_{k \in I(\overline R_\ell)} \overline E(k) \overline R_\ell(k)} \sum_{j\in X_L}  \left( \prod_{i \in \pi } \frac{d}{d \varepsilon_i} \right)  e^{\sum_{k=0}^{L} \overline  E(j_k) }\prod_{i=1}^L \overline R_i(j_{i-1}) \overline R_i(j_i) \\
% \\
% &+ \left[\left( \prod_{i \in \pi } \frac{d}{d \varepsilon_i} \right)  \prod_{\ell=1}^L
% e^{\sum_{k \in I(\overline R_\ell)} \overline E(k) \overline R_\ell(k)} \right] \sum_{j\in X_L}   e^{\sum_{k=0}^{L} \overline  E(j_k) }\prod_{i=1}^L \overline R_i(j_{i-1}) \overline R_i(j_i) \\
\end{align*} 
%where 
%\begin{align*}
% P(\gamma, \overline E) := \left[\left( \prod_{i \in \gamma } \frac{d}{d \varepsilon_i} \right)  \prod_{\ell=1}^L
% e^{\sum_{k \in I(\overline R_\ell)} \overline E(k) \overline R_\ell(k)} \right] \sum_{j\in X_L} \left( \prod_{i \in \pi \setminus \gamma } \frac{d}{d \varepsilon_i} \right)  e^{\sum_{k=0}^{L} \overline  E(j_k) } \prod_{i=1}^L \overline R_i(j_{i-1}) \overline R_i(j_i). 
% \end{align*} 
%Notice that for $\gamma = \emptyset$ we have that
%\begin{align*}
% P(\emptyset, \overline E) &=\prod_{\ell=1}^L
 %e^{\sum_{k \in I(\overline R_\ell)} \overline E(k) \overline R_\ell(k)} \sum_{j\in X_L} \left( \prod_{i \in \pi } \frac{d}{d \varepsilon_i} \right)  e^{\sum_{k=0}^{L} \overline  E(j_k) } \prod_{i=1}^L \overline R_i(j_{i-1}) \overline R_i(j_i)
 %\\
 %&=\prod_{\ell=1}^L \alpha_{\overline R_\ell }   e^{\sum_{k=0}^{L} \overline  E(j^\pi_k) } \prod_{i=1}^L \overline R_i(j^\pi_{i-1}) \overline R_i(j^\pi_i) = e^{\sum_{k=0}^{L} \overline  E(j^\pi_k) } \prod_{\ell=1}^L \alpha_{\overline R_\ell }    \prod_{i=1}^L \overline R_i(j^\pi_{i-1}) \overline R_i(j^\pi_i)
% \end{align*} 
where $j^\pi_i $ are the vertices in the path $\pi$. 

As a consequence we deduce that for sufficiently small values of $\varepsilon >0 $ we have that $\Delta_\delta c_{L,p} ( \overline {\mathcal K}, \overline E  ) \neq 0$. Hence the desired conclusion follows, i.e. we have that for sufficiently small values of $\varepsilon$ and of $\delta $ we have that $c_{L,p} ( \overline {\mathcal K}, \overline E  )  \neq 0$.

In particular we have proven that there exists a rate function $\overline {\mathcal K} $ that satisfies the assumption of the Theorem and that is such that 
\[
\varphi_{L, p} (t) = \frac{1}{(L+1)!}c_{L+1, p}(\overline{\mathcal K}, \overline E) t^{L+1} + o(t^{L+1}), \quad t >0
\]
where $c_{L+1, p}(\overline{\mathcal K}, \overline E) \neq 0$ and where  $\varphi$ is the solution of \eqref{eq:linearized odes} corresponding to the kinetic system $(\Omega , \mathcal R, \overline{\mathcal K })$. 
As a consequence we deduce that for every $T>0 $ we have $\sup_{t \in (0,T]} |\varphi_p(t) |>0$.  Hence the solution of the system of ODEs \eqref{ODEs signalling} induced by $(\Omega, \mathcal R, \overline{\mathcal K }) $ is such that the property 3 in Definition \ref{def:adaptation} holds, i.e. $ \sup_{ t >0} | n_p (t ) - N(p) | > 0$.  
\end{proof}
We now present an example of a closed kinetic system that is such that the graph $\mathcal G_\Omega $ is connected and the property $3$ holds generically, i.e. if property $3$ does not hold then it is possible to construct a perturbation of the reaction rates that is such that the perturbed kinetic system satisfies the property $3$, i.e. the system reacts to changes in the signal. The purpose of this example is also to illustrate the proof of Theorem \ref{thm:reaction} in a simpler situation. 
\begin{example}
    Consider the chemical network corresponding to the following reactions
    \[
    (1) \leftrightarrows (2)+(3), \quad (2) \leftrightarrows (3)+(4). 
    \]
    Let $R_1=(1,-1,-1,0)^T$ and $R_2=(0,1,-1,-1)^T$. Notice that the chemical network is connected. Moreover, we assume that 
    \[
    K_{- R_1} = K_{R_1} e^{E(1)- E(2) - E(3) }, \text{ and }     K_{- R_2} = K_{R_2} e^{E(2) - E(3) - E(4) } 
    \]
    for some vector $E \in \mathbb R^4$.
    Notice that in this way we assume that the detailed balance property holds by construction. In particular the energies of the reaction $R_1 $ and $R_2$ are given by 
    \[
    \mathcal  E (R_1) = E(1)- E(2) - E(3), \text{ and } \   \mathcal  E (R_2)= E(2) - E(3) - E(4). 
    \]
    To simplify the notation we indicate with $K_1 $ the rate $K_{R_1} $ and with  $K_2$ the rate $K_{R_2} $. 
    We can write the system of ODEs associated with the kinetic system 
    \begin{align*}
        \frac{d  n_1}{d t}&=K_1  \left( n_2 n_3 - e^{\mathcal E(R_1) } n_1  \right) \\
       \frac{d  n_2}{d t}&=K_1  \left( e^{\mathcal E(R_1) } n_1  -  n_2 n_3\right)   + K_2 \left(n_3 n_4 - e^{\mathcal E(R_2) } n_2  \right)       \\  
         \frac{d  n_3}{d t}&=K_1  \left( e^{\mathcal E(R_1) } n_1  -  n_3 n_3\right)   + K_2 \left(e^{\mathcal E(R_2) } n_2 - n_3 n_4  \right)       \\  
          \frac{d  n_4}{d t}&= K_2 \left(e^{\mathcal E(R_2) } n_2 - n_3 n_4  \right)       \\  
    \end{align*}
    We linearize the system around the steady state $ N=e^{- E} $. In particular we consider $n_k =(1+ \varphi_k ) e^{- E(k) } $ with $\varphi_1 (t) = t $. Hence
     \begin{align*}
       \varphi_1 (t) &= t  \\
 \frac{d  \varphi_2}{d t}&=K_1 e^{- E(3) } \varphi_1  + l.o.t. \\  
     \frac{d  \varphi_3}{d t}&=K_1 e^{- E(2) } \varphi_1  + l.o.t. \\    
          \frac{d  \varphi_4}{d t}&= K_2 e^{-E(3) }\left( \varphi_2 - \varphi_3  \right)     + l.o.t.    \\  
    \end{align*} 
As a consequence we obtain that $\varphi_2  \sim \frac{K_1}{2} e^{- E(3) } t^2 $ as $ t \to 0$ and $\varphi_3 =\frac{K_1}{2} e^{- E(2) } t^2 $ as $ t \to 0$. 
Therefore 
\[ 
\varphi_4 (t)\sim \frac{K_1 K_2 e^{- E(3)}}{2} \left( e^{- E(3)} - e^{-E(2)} \right) \quad \text{ as } t \to 0. 
\] 
As a consequence we have that $\varphi_4(t) \neq 0 $ for sufficiently small times, except if $E(3)=E(2)$, i.e. unless the reaction rates are fine tuned, we have that the concentration of substance $(4)$ changes in time of the signal concentration changes changes according to an admissible signal $f$.
\end{example}

\subsection{Stable adaptation in non-conservative systems that satisfy the detailed balance property}\label{sec:db and adaptation}
In this section we study whether a signalling system with the detailed balance property (or, more precisely, the underlying kinetic system satisfies the detailed balance property) can have the adaptation property in a stable manner, i.e. when the parameters are not fine tuned.  We have multiple results in this direction. 
First of all we prove that every system that satisfies the detailed balance property, is connected and that is non conservative and is such that the signal appears in one of the conservation laws, satisfies the adaptation property, unless the parameters are fine tuned. 

As a second result we prove that if a signalling system is such that the underlying kinetic systems is closed and its  conservation laws satisfy a suitable non-degeneracy assumption, then it does not have the adaptation property in a stable manner.  
The non-degeneracy assumption that we have to make on the conservation laws is a rather natural assumption. In particular, we assume that every conservation law can be written as the sum of conservation laws that contain both the product and the signal, i.e. $ \forall m \in \mathcal B  $  we have $m(1)\neq 0 $ and $m(p) \neq 0$, where $1 $ is the signal and $p $ is the product. We will discuss this assumption in detail in Section \ref{sec:adaptation in closed systems}. 

\begin{proposition} \label{prop:DB and adaptation}
Assume that the kinetic system $(\Omega, \mathcal R, \mathcal K) $ satisfies the detailed balance property, assume that the chemical network $(\Omega, \mathcal R )$ is connected and is such that $\mathcal M \neq \{ 0\} $, without loss of generality we assume that there exists a $m \in \mathcal M $ such that $m(1)>0$. 
Moreover assume that the chemical network $(\Omega, \mathcal R,  \mathcal K) $ is not conservative.  
Assume that the function $n_1=f$ is an admissible signal. 
    Then we have that  either \begin{itemize}
    \item there exists a $p \in \Omega \setminus \{ 1\} $ such that the network $(\Omega, \mathcal R, \mathcal K , n_1 (t) ) $ satisfies the adaptation property with respect to $p$; 
    \item or alternatively, we have that for every $\delta >0$ there exists a rate function $ \overline {\mathcal K} : \mathcal R \rightarrow \mathbb R_+$ satisfying the assumptions of Theorem \ref{thm:reaction} and such that  $(\Omega , \mathcal R , \overline{\mathcal K}, n_1(t) ) $ satisfies the adaptation property. 
    \end{itemize} 
\end{proposition}
\begin{proof}
Let $\overline n_1$ be such that $\lim_{t \to \infty } f(t)= \overline n_1 $. Then 
Proposition \ref{prop:property 1 adaptation} implies that there exists an energy $\overline E $ of the kinetic system $(\Omega , \mathcal R, \mathcal K ) $ such that $\lim_{ t \to \infty } n (t) = e^{- \overline E} $. This in particular implies that the property 1 in the Definition \ref{def:adaptation} holds. 

Since the chemical network $(\Omega , \mathcal R)$ is non-conservative there exists a $p \in \Omega \setminus \{1\} $ that is such that $m (p )=0$ for every $m \in \mathcal M$. 
As a consequence we deduce that all the energies $E$ of the kinetic system $(\Omega, \mathcal R, \mathcal K ) $ are such that $E(p)=\overline E$. Indeed, assume by contradiction that there exists two energies $E_1$ and $E_2 $ satisfying \eqref{eq:energy} and that are such that $E_1(p) \neq E_2(p)$. This implies that $\textbf{R}^T (E_1- E_2 ) =0$, hence that  $E_1- E_2 \in \mathcal M $. However since the kinetic system is non conservative in $p$ we have that every $m \in \mathcal M $ is such that $m (p) =0 $. This implies that $E_1(p)=E_2(p)$, a contradiction. 
As a consequence we have that every steady state of $(\Omega, \mathcal R, \mathcal K ) $ is such that $N (p)=e^{- \overline E(p)} $. Hence property 2. in the Definition of adaptation holds. 

Finally since the assumptions of Theorem \ref{thm:reaction} hold, also the third property necessary to have the adaptation property holds. 
\end{proof}

\subsection{Stable adaptation is impossible in closed signalling systems} \label{sec:adaptation in closed systems}
In this section we prove that if the conservation laws of a kinetic system satisfy a suitable non-degeneracy property, then the fact that the kinetic system satisfies the adaptation property implies that the kinetic system is not closed. 
We start the section by explaining the property that the conservation laws must satisfy. 

First of all let us consider the  basis  of the extreme rays $\mathcal B = \{ m_i \}_{i=1}^L \subset \mathcal M_+$ of $\mathcal M$. 
Consider $i \in \Omega $ and let us define the vector $m^{(i)} \in \mathbb R_*^L $ that is such that
\[
m^{(i)}(j) := m_j(i). 
\]

\begin{definition}[$\mathcal M $-connectivity]
Let $(\Omega , \mathcal R ) $ be a chemical network and let $\mathcal M $ be the associated set of  conservation laws. 
We say the two elements $i, j \in \Omega $ are $\mathcal M $-connected if $ \langle m^{(i)} , m^{(j)}\rangle \neq 0 $. 
We say that the chemical network $(\Omega, \mathcal K ) $ is $\mathcal M $-connected if every $i , j \in \Omega $ are $\mathcal M $-connected. 
\end{definition}

\begin{example}[Adaptation property when $\mathcal M $-connectivity fails]
Consider the chemical system induced by the following reactions
\[
(1) + (3) \leftrightarrows (2) + (4),\quad (1) \leftrightarrows (2), \quad (3) \leftrightarrows (4). 
\]
Notice that the basis of extreme rays $\mathcal B $ is given by $\mathcal B =\{ (1,1,0,0)^T , (0,0,1,1)^T \} $. In particular we have that $m^{(1)} =(1,0)^T $ and that $m^{(4)} =(0,1)^T  $, therefore $ \langle m^{(1)} , m^{(4)} \rangle= 0 $. Hence the substance $(1)$ and the substance $(4)$ are not $\mathcal M $ connected. 
Moreover notice that the chemical network is conservative. We can associate to the network a rate function that satisfies the detailed balance condition. 
We can assume that $n_1=f$ is an admissible signal. Then by Proposition \ref{prop:property 1 adaptation} we know that the system converges to a steady state $N$ that satisfies \eqref{cons laws when fluxes2}, hence in particular we have that 
\[
N_3 + N_4 - n_3(0) - n_4(0)=0 \ \text{ and } \  \frac{N_3}{N_4}=\frac{n_3(0)}{n_4(0)}.
\]
This implies that $N_4 $ does not depend on the signal $n_1=f$. Since we have that Theorem \ref{thm:reaction} guarantees that property 2. in the Definition of adaptation holds we deduce that this kinetic system satisfies the adaptation property even if it is closed. 
\end{example}

\begin{theorem} \label{thm:no adaptation in closed systems}
    Let $(\Omega,\mathcal R, \mathcal K, n_s )$ be a signalling kinetic system is $\mathcal M $-connected. Assume that the adaptation property holds, then 
    \begin{itemize} 
    \item either the kinetic system $(\Omega,\mathcal R, \mathcal K )$ is not closed, 
    \item or for every $\delta >0 $ there exists a map $ \mathcal K_\delta : \mathcal R \rightarrow \mathbb R_+$ such that
            \[
            \|\mathcal K- \mathcal K_\delta \| < \delta 
            \]
            and such that $(\Omega, \mathcal R, \mathcal K_\delta)$ does not satisfy the adaptation property. 
    \end{itemize}
\end{theorem}
\begin{proof}
Let us assume that $(\Omega,\mathcal R, \mathcal K )$ is closed, hence it satisfies the detailed balance property. 
Notice that Proposition \ref{prop:property 1 adaptation} together with the property 1 of Definition \ref{def:adaptation} implies that there exists an energy $E$ of $(\Omega, \mathcal R, \mathcal K )$ that is such that $n(t) \rightarrow e^{- E }$ as $ t \to \infty$. Let $\mathcal B = \{ m_k\}_{k=1}^L $. 
Using \eqref{cons laws when fluxes2} we deduce that for every $k \in \{ 1, \dots, L \} $ it holds that 
\[
m_k^T e^{-E } - m_k^T e^{- E_0 } =m^{(1)}_k  \overline J_1^F 
\]
where $n_0=e^{- E_0} $. 
Since $E $ and $E_0 $ are two energies of the same kinetic system we have that $E=E_0 + \sum_{i =1}^L m_i \eta_i $ for some chemical potentials $\{ \eta_i \}_{i=1}^L $ with $\eta_i \in \mathbb R$. 
Hence the above equality reduces to 
\[
m^{(1)}_k  \overline J_1^F = \sum_{j \in \Omega } m_k^{(j)} e^{- E_0 (j) } \left( e^{\sum_{i =1}^L m^{(j)}_i \eta_i} -1 \right). 
\]
Let us stress that the chemical potentials $\eta_i $ depend on $\overline J_1^F$.
Therefore differentiating the above equality with respect to $\overline J_1^F$ we deduce that the property 2 in Definition \ref{def:adaptation} implies that
\[
m^{(1)}_k = \sum_{j \in \Omega } m_k^{(j)} e^{- E (j) } \sum_{i=1}^L m_i^{(j)} \frac{\partial \eta_i}{\partial_{\overline J_1^F}}. 
\]
This equality can be rewritten as 
\[
m^{(1)}= M F(E) M^T \xi 
\]
where $M \in \mathbb R^{L \times N } $ with $M_{ij}=m^{(j)}_i$, $F(E)=\operatorname{diag} \{ e^{-E(i)} \}_{i=1}^N$ and where $\xi \in \mathbb R^L $ is defined as $ \xi(i)=\frac{\partial \eta_i}{\partial_{\overline J_1^F}}$.
On the other hand, since the adaptation property holds, we have that 
\[
m_k^{(N)} e^{-E(N) } - m_k^{(N)} e^{- E_0(N)} =0
\]
where $N=p$ is the product. 
As a consequence, differentiating with respect to $\overline J_1^F$ also this equation, we deduce that the adaptation property implies that 
\begin{align} \label{adapt proof eq}
    \begin{cases}
        M F(E) M^T \xi &=  m^{(1)}\\
         \langle m^{(N)}, \xi\rangle &=0.
    \end{cases}
\end{align}
Let us define the matrix $\mathcal D (\zeta) \in \mathbb R^{N \times N} $ as  
\[
\mathcal D (\zeta):= \sum_{j=1}^N \zeta_j \left( m^{(j)} \otimes m^{(j)} \right)= M F(E) M^T,
\]
where we use the notation $\zeta = e^{- E } $.

Notice that since the kinetic system $(\Omega, \mathcal R, \mathcal K ) $ is conservative, the matrix $\mathcal D(\zeta) $ is positive definite by definition, hence we can diagonalize it, i.e. 
\[
\mathcal D (\zeta)= \sum_{j=1}^N \lambda_j (\zeta) \left( \eta_j (\zeta) \otimes \eta_j (\zeta) \right) 
\]
where $\{ \eta_j(\zeta) \} $ is an orthonormal basis of $\mathbb R^N $ and $\{ \lambda_j (\zeta) \} $ are the eigenvalues of $\mathcal D(\zeta) $. Notice that since $\mathcal D(\zeta)$ is positive definite we have that for every $j \in \{ 1, \dots, N\} $ we have that $\lambda_j (\zeta) >0$. 
Moreover,  the inverse of $\mathcal D (\zeta)$ is given by 
\[
\mathcal D (\zeta)^{-1} = \sum_{j=1}^N \frac{1}{\lambda_j (\zeta)} \left( \eta_j (\zeta) \otimes \eta_j (\zeta) \right). 
\]
This, together with \eqref{adapt proof eq} implies that 
$\xi= \mathcal D (\zeta)^{-1}m^{(1)}$. 

Using the fact that $ \langle m^{(N)} , \xi \rangle  =0$ and that the matrix $\mathcal D(\zeta)^{-1} $ is symmetric we deduce that the property of adaptation implies that
\[
 \langle  m^{(N)},  \mathcal D (\zeta)^{-1}m^{(1)} \rangle = \langle m^{(1)}, \mathcal D (\zeta)^{-1} m^{(N)} \rangle =0 . 
\]

We now define the following equivalence relation. We say that $i \sim j $ if and only if we have that either 
\[
\langle m^{(i)} ,\mathcal D (\zeta)^{-1} m^{(k)} \rangle \neq 0 
\]
or, alternatively for every $\delta >0$ there exists a $ x \in B_\delta (\zeta) \subset \mathbb R^L$ such that 
\[
 \langle m^{(i)} ,\mathcal D (x)^{-1} m^{(k)} \rangle \neq 0, \quad \forall z \in B_\delta (\zeta). 
\]
The relation $\sim $ is an equivalence relation. Indeed, we have that for every $i \in \Omega$ it holds that $(i) \sim (i) $, because $\mathcal D(\zeta) $ is positive definite. Moreover the symmetry of $\mathcal D^{-1}(\zeta ) $ implies that $(i) \sim (j) $ iff $(j) \sim (i) $. Finally we prove the transitive property. We have to prove that the fact that $(i) \sim (j)  $ and $(j) \sim (k) $ implies $(i) \sim (k) $.
Indeed, notice that by the definition of $\mathcal D(\zeta) $ we have 
\[
 \frac{\partial }{\partial_{\zeta_j} }\mathcal D (\zeta)^{-1} = - \mathcal D (\zeta)^{-1}  \left( m^{(j)} \otimes m^{(j)} \right) \mathcal D (\zeta)^{-1}. 
\]
Therefore, if we assume that $\langle m^{(i)} ,  \mathcal D (\zeta)^{-1}   m^{(j)} \rangle  \langle  m^{(j)} , \mathcal D (\zeta)^{-1}  m^{(k)} \rangle \neq 0$, then 
\begin{align*}
\frac{\partial }{\partial_{\zeta_j} } \langle m^{(i)} ,\mathcal D (\zeta)^{-1} m^{(k)} \rangle  &= \langle m^{(i)} , \frac{\partial }{\partial_{\zeta_j} } \mathcal D (\zeta)^{-1} m^{(k)} \rangle = -   \langle m^{(i)} ,  \mathcal D (\zeta)^{-1}  \left( m^{(j)} \otimes m^{(j)} \right) \mathcal D (\zeta)^{-1}  m^{(k)} \rangle \\
&= -   \langle m^{(i)} ,  \mathcal D (\zeta)^{-1}   m^{(j)} \rangle  \langle  m^{(j)} , \mathcal D (\zeta)^{-1}  m^{(k)} \rangle \neq 0. 
\end{align*}
This in particular implies that for every $\delta >0 $ there exists a $ z \in B_\delta (\zeta) $ such that $\langle m^{(i)} ,\mathcal D (z)^{-1} m^{(k)} \rangle \neq 0$. 
Hence $(i) \sim (k) $. 
Assume instead that $\langle m^{(i)} ,  \mathcal D (\zeta)^{-1}   m^{(j)} \rangle  \langle  m^{(j)} , \mathcal D (\zeta)^{-1}  m^{(k)} \rangle =0$. Since $(i) \sim (j) $ and $(j) \sim (k) $ we know that for every $\delta >0 $ there exists a $ z \in B_\delta (\zeta) $ such that \[
\langle m^{(i)} ,  \mathcal D (z)^{-1}   m^{(j)} \rangle  \langle  m^{(j)} , \mathcal D (z)^{-1}  m^{(k)} \rangle \neq 0.
\]
The above computation shows that for every $\delta >0 $ there exists a $x \in B_\delta (z)$ such that $\langle \mathcal D (x)^{-1}  m^{(k)} \rangle \neq 0$. 
Notice that this implies that $(i) \sim (k)$.

We denote with $\mathbb M_1$ the equivalence class (corresponding to the equivalence relation $\sim $) that contains $m^{(1)}$. 
We want to prove that the fact that the kinetic system is $\mathcal M $-connected implies that, upon small perturbation of the reaction rates, we have $\mathbb M_1=\{ m^{(i)} \}_{i \in \Omega } $.
To this end we prove that if there exist more than one equivalence class induced by $\sim $, then the kinetic system $(\Omega, \mathcal R, \mathcal K ) $ is not $\mathcal M $-connected. 
So let us assume that there exists at least two equivalence classes on the set $\{ m^{(i)} \}_{i \in \Omega } $ induced by $\sim$. We define the following two sets
\[
W_1 := \{ y \in \mathbb R^L :\langle y, \mathcal D(\zeta)^{-1} m^{(k)}\rangle =0, \quad  \forall m^{(k)} \in \mathbb M_1  \} 
\]
and 
\[ 
W_2:= \{ y \in \mathbb R^L :\langle y, \mathcal D(\zeta)^{-1} m^{(k)}\rangle =0, \quad  \forall m^{(k)} \in \{ m^{(i)} \}_{i \in \Omega } \setminus  \mathbb M_1  \}.
\]
We now want to prove that $W_1 \cap W_2 = \{ 0\} $. 
To this end recall that, since the rank of the matrix $M $ is $L$, we have that there exists a subset $\{ m^{(i_j)} \}_{j=1 }^L  $ of $\{ m^{(i)} \}_{i \in \Omega } $ of linearly independent vectors. As a consequence the corresponding  vectors $\{ \xi^{(i_j)} \}_{j=1 }^L $ with $\xi^{(i_j)} := \mathcal D(\zeta)^{-1} m^{(i_j)}$ are also linearly independent. Indeed assume that $\{ \alpha_{j}\}_{j =1}^L $ is such that 
\[
0=\sum_{j =1}^L \alpha_j \xi^{(i_j)}=  \mathcal D(\zeta)^{-1} \sum_{j =1}^L \alpha_j m^{(i_j)}
\]
Since $\mathcal D(\zeta)^{-1} $ is invertible this implies that $ \sum_{j =1}^L \alpha_j m^{(i_j)}=0$. Hence $\alpha_j=0$ for every $j \in \{ 1, \dots, L \} $. 
As a consequence we deduce that if $y \in W_1 \cap W_2 = \{ 0\} $ then $\langle y, \xi^{(i_j)} \rangle=0$ for every $j \in \{ 1, \dots, L \} $. Hence $y=0$. 

We now prove that $W_1 \oplus W_2 = \mathbb R^L $. 
Without loss of generality assume that the number of vectors of the basis $\{ \xi^{(i_j)} \}_{j=1}^L $ contained in the class $\mathbb M_1 $ is given by $\gamma_1 $. As a consequence we have that $\dim W_1 = L- \gamma_1$ while the dimension of $W_2 = \gamma_1 $. As a consequence, since $W_1 \cap W_2 = \{ 0\} $ we deduce that $\dim (W_1+W_2)=L $, hence the desired conclusion follows and we obtain that $W_1 \oplus W_2 = \mathbb R^L$.  
 
Let us define $\Omega_1 =\{ k: m^{(k)} \in W_1 \} $ and $\Omega_2 =\{ k: m^{(k)}  \in W_2 \} $. Then by the definition of conservation law we have that
\[
0 = \sum_{j \in \Omega }   R(j) m^{(j) } = \sum_{j \in \Omega_1 }   \pi_{\Omega_1} R(j) m^{(j) } + \sum_{j \in \Omega_2 } \pi_{\Omega_2}  R(j) m^{(j) }. 
\]
Here we are using the notation $ \pi_A v \in \mathbb R^{|A|} $ 
\[
\pi_A v :=\left(  v(i)\right)_{i \in A}=(v( \min(A)), \dots,  v(\max A)),  
\]
where $A \subset \Omega $. 
This implies that 
\[ 
 \sum_{j \in \Omega_1 }   \pi_{\Omega_1} R(j) m^{(j) } =- \sum_{j \in \Omega_2 }   \pi_{\Omega_2}  R(j) m^{(j) } \in W_1 \cap W_2 = \{ 0\} . 
 \]
Notice that this implies that every vector $m_k$ that belongs to the extremal basis $\mathcal B $ can be written as the sum of two conservation laws, i.e. $m_k=(\pi_{\Omega_1} m_k, 0 ) + (\pi_{\Omega_2} m_k, 0 )$.

However, since $\mathcal B $ is a basis of extreme rays of $\mathcal M_+$ we have that this implies either that $\pi_{\Omega_1} m_k=0$ or that $\pi_{\Omega_2} m_k=0$. Hence this implies that $(\Omega , \mathcal R, \mathcal K ) $ is not $\mathcal M $-connected. 
 This concludes the proof of the fact that $\mathcal M $-connectivity implies that there exists a unique equivalence class induced by $\sim $.

In particular this implies that either 
\[
\langle \mathcal D(\zeta )  m^{(1)}, m^{(N) } \rangle \neq 0 
\]
or alternatively, if $\langle \mathcal D(\zeta )  m^{(1)}, m^{(N) } \rangle = 0$, then for every $\delta >0$ there exists a $z\in B_\delta (\zeta)$ such that
\[
\langle \mathcal D(z)  m^{(1)}, m^{(N) } \rangle \neq 0. 
\]
This implies the statement of the theorem. Indeed recall that $\zeta = e^{-E}$. Hence for any $\delta >0 $ we can construct the rate function $\mathcal K_\delta : \mathcal R \rightarrow \mathbb R_*$ in such a way that
\[
\frac{\mathcal K_\delta (-R) }{\mathcal K_\delta (R) }= e^{- \sum_{i \in \Omega } \overline E(i) R(i) } \quad  \forall R \in \mathcal R_s. 
\]
where $z=e^{- \overline E}$. Then the kinetic system $(\Omega, \mathcal R, \mathcal K_\delta)$ does not satisfy the adaptation property, because \eqref{adapt proof eq} fails. 
\end{proof}
\section{Some models exhibiting the adaptation property} \label{sec:examples}
In this section we revise some models of adaptation that can be found in the literature and we explain their relation with the results presented in this paper and some of the reults in \cite{franco2025reduction}. 

\subsection{Some classical models of adaptation with fine tuned parameters}
We start analysing the classical adaptation model proposed in \cite{segel1986mechanism} by Segel, Goldbeter, Devreotes, and Knox. 
One of the goals of this section is to illustrate that, even if the kinetic system in \cite{segel1986mechanism} is closed and satisfies the adaptation property, the results of that paper are consistent with Theorem \ref{thm:no adaptation in closed systems}. Indeed, the model in \cite{segel1986mechanism} satisfies the adaptation property only for fine-tuned parameters. 
A second goal of this section is to review briefly \cite{walz1987consequences}. This paper is interesting for us because it studies different enlarged kinetic systems that, when reduced in a suitable manner, give rise to the kinetic system studied in \cite{segel1986mechanism}. The results obtained in \cite{walz1987consequences} are consistent with our results on the completion and reduction of kinetic systems studied in \cite{franco2025reduction}.

The substances of the network in \cite{segel1986mechanism} are two type of receptors $D $ and $R$, the ligands  $L $ and two types of complexes ligand receptor $X=[LR]$ and $Y=[LD]$.
The signal in this case is the ligand. 
The set of the reactions taking place in the signalling system is 
\[
(L) + (R) \underset{k_{-r}}{\overset{k_r}{\leftrightarrows}}(X), \quad (L) + (D) \underset{k_{-d}}{\overset{k_d}{\leftrightarrows}} (Y), \quad (R) \underset{k_{-1}}{\overset{k_1}{\leftrightarrows}}  (D), \quad (X) \underset{k_{-2}}{\overset{k_2}{\leftrightarrows}}   (Y). 
\]

We have two conserved quantity in this system, i.e. the number of receptors is conserved and the number of ligands is conserved
\[
R+D+X+Y=R_0\ \text{ and }\  L+X+Y=L_0. 
\]
In particular this imply that the system is conservative. 
The reactions rates in \cite{segel1986mechanism} are assumed to satisfy the detailed balance property (see equation (29) in their paper). 
As a consequence, the kinetic system is closed.
Notice moreover that the conservation laws satisfy the $\mathcal M $-connectivity property. 
It is proven in \cite{segel1986mechanism} that this signalling kinetic system satisfies the adaptation property. However this is proven for fine-tuned parameters. Therefore the result in \cite{segel1986mechanism} is in agreement with Theorem \ref{thm:no adaptation in closed systems} in our paper, where we state that, unless the parameters are fine tuned or the conservation quantities have a specific structure, closed kinetic systems do not satisfy the adaptation property. 

Generalizations of the kinetic system above were also considered in \cite{walz1987consequences}, where different ways to complete the kinetic system studied in \cite{segel1986mechanism} are considered.
In particular we briefly review two of these completions. 
The first completion consist in assuming that the reactions $(R) \leftrightarrows (D) $ and $(X) \leftrightarrows (Y) $ are obtained starting from reactions involving two additional substances, for instance the substances $(A), (B) $, that in the reduced system studied in \cite{segel1986mechanism} are assumed to be constant in time. 
Then the completed reactions are of the form $(R) + (A) \leftrightarrows (D)+(B)  $ and $(X) + (A)  \leftrightarrows (Y) + (B) $. 
It is shown in \cite{walz1987consequences} that the detailed balance property of the kinetic system in \cite{segel1986mechanism} follows from the detailed balance of this completed systems.
The reason is that the cycles of the completed system and the kinetic system in \cite{segel1986mechanism} are the same, i.e. we have that the space of the cycles for both the systems is 
\[
 \operatorname{span} \{ (1,-1,-1,1)  \}. 
 \]
 This is consistent with Theorem 5.4 in \cite{franco2025reduction}. Indeed this theorem states that if we consider a system with the detailed balance property and we freeze the concentration of certain substances in the network in such a way that the reduced network (involving only the non constant concentrations) and the original network have the same cycles, then the detailed balance of the reduced network follows directly from the detailed balance of the original network. 

The second completion consist in assuming that the reactions $(R) \leftrightarrows (D) $ and $(X) \leftrightarrows (Y) $ are obtained as the reduction of reactions involving four additional substances, for instance the substances $A, B, W,Z $. Then the completed reactions are of the form $(R) + (A) \leftrightarrows (D)+(B)  $ and $(X) + (W)  \leftrightarrows (Y) + (Z) $. As mentioned in \cite{walz1987consequences}, the detailed balance property of the kinetic system in \cite{segel1986mechanism} can only be obtained assuming that the concentration of the additional substances $(A), (B), (W),(Z) $ are chosen at equilibrium values. Indeed, it is easy to notice that the completed kinetic system does not have cycles while the kinetic system in \cite{segel1986mechanism} has cycles. Notice that this is consistent with Proposition 5.6 in \cite{franco2025reduction} where it is proven that if the reduced system has more cycles than the completed systems then the only way to have that the detailed balance property of the reduced system holds in a robust manner is to choose the frozen concentrations at equlibirum values. 

\subsection{The Barkai-Leibler model of robust adaptation for bacterial chemotaxis}
In this section we briefly review a linear version of another classical model of robust adaptation, i.e. the Barkai-Leibler model of bacterial chemotaxis proposed in \cite{barkai1997robustness}. 
In contrast with the model in \cite{segel1986mechanism}, this model satisfies the property of adaptation in a robust manner. 

As we will see, the kinetic system studied in \cite{barkai1997robustness} is not a mass action kinetic system. The goal of this section is to give an idea on how this model can be derived starting from kinetic systems with mass action kinetics. 
To this end we need to make suitable assumptions on the speed of certain reactions taking place in the network. 

 A linear version of the model in \cite{barkai1997robustness} is given by the following system of ODEs
\begin{align} \label{adaptation BL}
\frac{dX}{dt}&= Y - (1-c) X + f(t) \\
\frac{dY}{dt}&= 1-c X, \nonumber
\end{align}
for a suitable $c>0$.
Here $X$ is the quantity of active receptors and $f$ is the function describing the evolution of the signal, i.e. the evolution of the concentration of ligands. Finally $Y$ is the response regulator protein. The product of this system will be the quantity
of active receptors. Notice that the system of ODEs \eqref{adaptation BL} does not correspond to a mass action kinetics. We briefly explain, without entering into the technical details, how this model can be obtained as a limit of mass action kinetic systems. 
Indeed, the system of ODEs above can be obtained from the following mass action kinetic system 
\begin{align}
   & X \overset{1}\rightarrow \emptyset, \quad  Y \overset{1}\rightarrow X, \quad   S\overset{1}\rightarrow X, \quad \emptyset  \overset{1}\rightarrow Y \label{non enzymes} \\
   & E+Y  \underset{k_{-1}}{\overset{k_1}\leftrightarrows} [yE],\quad  [yE]+ X\underset{k_{-2}}{\overset{k_2}\leftrightarrows}[xyE], \quad [xyE]\overset{\lambda }\rightarrow E+ P \label{kinetic system for mass action adaptation fine tuned}. 
\end{align}
Here $S$ represents the signal, $P$ the product of the enzymatic reactions, $E$ an enzyme, $[yE] $ is the complex formed by the enzyme and regulator protein, while $[xy E] $ is the complex formed by the enzyme, the regulator protein and the active receptor.
The system of ODEs corresponding to the enzymatic reactions in \eqref{kinetic system for mass action adaptation fine tuned} are of the following form 
\begin{align*} 
\frac{dX}{dt}&= k_{-2} [xy E]-k_2 X [yE] \\
\frac{dY}{dt}&= -k_{1} E Y + k_{-1} [yE] \\ 
\frac{dE}{dt}&= -k_{1} E Y + k_{-1} [yE] + \lambda [xyE] \\
\frac{d[yE]}{dt}&= k_{1} E Y - k_{-1} [yE] +k_{-2} [xy E]-k_2 X [yE]\\
\frac{d[xyE]}{dt}&= - k_{-2} [xy E]+ k_2 X [yE] -  \lambda [xyE]\\
\frac{dP }{dt}&=   \lambda [xyE].
\end{align*} 

We then assume that the reaction $ E+Y  \underset{k_{-1}}{\overset{k_1}\leftrightarrows} [yE]$ and the reaction $ [yE]+ X\underset{k_{-2}}{\overset{k_2}\leftrightarrows}[xyE] $ are very fast compared to the other reactions. This implies that the concentration of $Y$, $X$ and of $[yE] $ are essentially constant, as they reach steady state values very quickly. 
Therefore from the equations for the evolution of $Y$, $X$ and of $[yE] $ we deduce that 
\[
 [yE]= \alpha_{1} E Y, \ \text{ and } \  [xy E]=\alpha_2 X [yE] = \alpha_1 \alpha_2 EY X 
\]
where $\alpha_1 = \frac{k_1}{k_{-1} } $ and $\alpha_2 = \frac{k_2}{k_{-2} } $. 
Now notice that the system of ODEs above have a conserved quantity, i.e. we have that 
\[
E+ [yE] + [xy E] =C_0. 
\]
From this we deduce that
\[
E+ \alpha_{1} E Y+ \alpha_1 \alpha_2 EY X =C_0,
\]
hence $ E=\frac{C_0}{1+ \alpha_1 Y+\alpha_1 \alpha_2 YX}$ and therefore 
\[
[xy E]= \frac{ C_0  \alpha_1 \alpha_2 Y X}{1+ \alpha_1 Y+\alpha_1 \alpha_2 YX}. 
\]
Assume now that $\alpha_2 \ll 1$ and that $\lambda \alpha_2 \approx 1 $ (hence $\lambda \gg 1$) and that $X\approx 1 $ and $Y\approx 1 $. Then we deduce that
\[
\frac{dP }{dt}=   \lambda [xyE] \approx (\lambda \alpha_2) \frac{ C_0  \alpha_1 Y X}{1+ \alpha_1 Y+\alpha_1 \alpha_2 YX}\approx \lambda \alpha_2 \frac{ C_0   Y X}{\frac{1}{\alpha_1}+ Y}.
\]
If in addition we assume that $\alpha_1 \gg 1 $ we obtain that 
\[
\frac{dP }{dt}\approx \lambda \alpha_2 C_0 X. 
\]
Under the assumptions above the enzymatic reactions \eqref{kinetic system for mass action adaptation fine tuned} can be reduced to the reaction 
\[
Y \rightarrow P. 
\]
where the rate of the reaction depends on $X$ and on the total number of enzymes, i.e. on $C_0$, more precisely the rate of the reaction is just $\lambda \alpha_2 C_0 X $. 
As a consequence, taking into account also of \eqref{non enzymes} we obtain the system of ODEs \eqref{adaptation BL}, where $c= \lambda \alpha_2 C_0$. 

\subsection{An adaptation model of gene expression}
In this section we study one of the models of adaptation considered in \cite{ferrell2016perfect}. The model that we are going to discuss, as well as all the models in \cite{ferrell2016perfect}, are one directional networks. Instead in this paper we mostly deal with bidirectional chemical networks. 
In this section we clarify the relation between bidirectional kinetic systems and one directional kinetic systems. More precisely we want to explain heuristically that a one directional kinetic system can be obtained as the limit of the bidirectional kinetic systems that do not satisfy the detailed balance property and that are such that their "lack of detailed balance tends to infinity". 
Finally we explain how it is possible to write a completion for one of the models studied in \cite{ferrell2016perfect}. 

Consider a bidirectional network $(\Omega , \mathcal R) $.
We know that we can always construct reaction rates that satisfy the detailed balance property, indeed to this end it is enough to select a free Gibbs energy $E$. Then the rates that satisfy the detailed balance property can be defined as 
\[ 
\log\left( \frac{K^{(DB)}_{-R}}{K^{(DB)}_{R}} \right) = \sum_{i \in \Omega } R(i) E(i) 
\]
for every $R \in \mathcal R$. 
On the other hand, assume that $\{ K_R\}_{R \in \mathcal R} $ are a set of rates that do not satisfy the detailed balance. 
Then we can define a vector $\Delta \in \mathbb R^{r/2}$ that is such that 
\[ 
\log\left( \frac{K_{-R}}{K_{R}} \right) =\Delta(R)+  \sum_{i \in \Omega } R(i) E(i) \ \iff \ K_{-R}  = K_{R} e^{\Delta(R)}  e^{ \sum_{i \in \Omega } R(i) E(i)}. 
\]
The vector $\Delta $ can be thought as the measure of the lack of detailed balance. Assume that for every $R \in \mathcal R $ we have that $\Delta (R) \gg 1 $, then we have that $K_{-R} \gg K_R $. Hence we can approximate the bidirectional kinetic system with a one directional kinetic system. 

Let us then consider the following chemical network, which is the bidirectional version of the chemical network in Figure 5 in \cite{ferrell2016perfect}
\begin{equation} \label{k system adapt with cycles}
(1)+ (2) \leftrightarrows (3), \ (2) \leftrightarrows \emptyset, \  (3) \leftrightarrows (4), \ (4)+(5) \leftrightarrows (6), \ (5) \leftrightarrows \emptyset, \  (2) \leftrightarrows (5). 
\end{equation}

We assume that the chemical network \eqref{k system adapt with cycles} is endowed with reaction rates that do not satisfy the detailed balance property. We stress that it is shown in \cite{ferrell2016perfect} that this model satisfies the adaptation property. This is consistent with our results, indeed, since the kinetic system does not have detailed balance, then it can satisfy the adaptation property. 
We now explain that this network can be obtained as a reduction of a closed kinetic system in which some substances are assumed to be constant in time.
Hence we complete the kinetic system following the proof of Theorem 6.3 in \cite{franco2025reduction}.
First of all notice that we can add four substances $(7), (8), (9) , (10)$ to the network. 
Then we can consider the following system of extended reactions. 
\begin{equation} \label{k system adapt with cycles no sinks and sources}
(1)+ (2) \leftrightarrows (3), \ (2) +(7)\leftrightarrows (8), \  (3) \leftrightarrows (4), \ (4)+(5) \leftrightarrows (6), \ (5) +(9)\leftrightarrows (10), \  (2) \leftrightarrows (5). 
\end{equation}
Notice that the kinetic system \eqref{k system adapt with cycles no sinks and sources} has no cycles, hence for every choice of rates is satisfies the detailed balance property. 
Moreover it is easy to check that $ m =(1,3,2,2,3,5,4,1,4)$ is a conserved quantity. Hence \eqref{k system adapt with cycles no sinks and sources} is a closed completion.

\vspace{1cm}
 
   \textbf{Acknowledgements:} The authors gratefully acknowledge the support by the Deutsche Forschungsgemeinschaft (DFG) through the collaborative research centre "The mathematics of emerging effects" (CRC 1060, Project-ID 211504053) and Germany's Excellence StrategyEXC2047/1-390685813.
The funders had no role in study design, analysis, decision to publish, or preparation of the manuscript.

\bibliographystyle{siam}

\bibliography{References}
\end{document}